\documentclass[12pt]{amsart}
\sloppypar
\usepackage{latexsym}
\usepackage{a4}
\usepackage{amssymb}
\usepackage{amsbsy}
\newcommand{\0}[3]{\pi^{(#1)}_{#2,#3}}
\newcommand{\Z}{\mathbb{Z}}

\numberwithin{equation}{section}
\newcommand{\AP}{\mathrm{(AP)}}
\newcommand{\C}{\mathrm{Comp}}  %
\newcommand{\Comp}{\C}

\newcommand{\join}{\lor}
\newcommand{\meet}{\land}
\renewcommand{\epsilon}{\varepsilon}
\newcommand{\eps}{\epsilon}
\newcommand{\ab}[1]{{\mathbf{#1}}}
\newcommand{\ob}[1]{{\mathbb{#1}}}
\newcommand{\A}{\ab{A}}
\newcommand{\vb}[1]{\bar{#1}}
\newcommand{\ABar}{\bar{\ab{A}}}
\newcommand{\AI}{\ABar|_{I}}
\newcommand{\N}{\mathbb{N}}
\newcommand{\setsuchthat}{\mid}
\newcommand{\Pol}{\mathrm{Pol}}
\newcommand{\Con}{\mathrm{Con}}
\newcommand{\ord}{\mathrm{ord}}
\newcommand{\Clo}{\mathrm{Clo}}

\newcommand{\id}{\mathrm{id}}

\newcommand{\IntI}{\mathbb{I}}
\newcommand{\algop}[2]{( {#1}, {#2} )}
\theoremstyle{plain} \newtheorem{thr}{Theorem}[section]
\theoremstyle{plain} \newtheorem{thm}[thr]{Theorem}
\theoremstyle{plain} \newtheorem{lem}[thr]{Lemma}
\theoremstyle{plain} \newtheorem{pr}[thr]{Proposition}
\theoremstyle{plain} \newtheorem{pro}[thr]{Proposition}
\theoremstyle{plain} 
\theoremstyle{plain} \newtheorem{que}[thr]{Question}
\theoremstyle{definition} \newtheorem{df}[thr]{Definition}
\theoremstyle{definition} \newtheorem{de}[thr]{Definition}

\title[Finite generation of congruence preserving functions]{Finite generation of congruence preserving functions}
\author{Erhard Aichinger}
\address{Institut f\"ur Algebra \\
Johannes Kepler Universit\"at Linz \\
4040 Linz, Austria}
\email{erhard@algebra.uni-linz.ac.at}

\author{Marijana Lazi\'c}
\address{Department of Mathematics and Informatics \\
Faculty of Sciences \\
University of Novi Sad \\
Trg Dositeja Obradovi\'{c}a 4 \\
21000 Novi Sad, Serbia \& Institut f\"ur Algebra \\
Johannes Kepler Universit\"at Linz \\
4040 Linz, Austria}
\email{marijana.lazic@dmi.uns.ac.rs}

\author{Neboj\v{s}a Mudrinski}
\address{Department of Mathematics and Informatics \\
Faculty of Sciences \\
University of Novi Sad \\
Trg Dositeja Obradovi\'{c}a 4 \\
21000 Novi Sad, Serbia \& Institut f\"ur Algebra \\
Johannes Kepler Universit\"at Linz \\
4040 Linz, Austria}
\email{nmudrinski@dmi.uns.ac.rs}

\thanks{Supported by Austrian research fund FWF P24077 and Research Grant 174018 of the Ministry of Science and Education
of the Republic of Serbia.}


\begin{document}

\bibliographystyle{amsalpha}

\begin{abstract}
We investigate when the clone of congruence preserving functions is finitely generated. We obtain a full description for
all finite $p$-groups, and for all finite algebras with Mal'cev term and simple congruence lattice. The characterization for $p$-groups
allows a generalization to a large class of expansions of groups.
\end{abstract}
\date{\today}
\maketitle

\section{Motivation} \label{sec:mot}
With each algebraic structure, one can associate several sets of finitary operations that contain  a lot of structural
information. One such set is the set of \emph{congruence preserving operations}
of the algebra.
Among these congruence preserving operations, we find all basic operations, all constant operations, 
all term operations, and all polynomial operations of the algebra. 
However, many algebras admit congruence preserving operations that
are not polynomial, and in contrast to polynomial operations,
the set of congruence preserving operations has no obvious set of
generators from which all congruence preserving operations can be composed. 
For example, 
the set of congruence preserving operations of the group
$\Z_4 \times \Z_2$ has
no finite set of generators at all \cite{Ai:2ACA}. In the present paper, we will investigate for
which finite algebras all congruence preserving operations can be generated from
a finite subset of such operations; in this case, we say that the clone of congruence
preserving functions is \emph{finitely generated}. So the main question
that we consider is:
\begin{que} \label{que:main}
   We are given a finite algebra $\ab{A}$. Is the clone of congruence
   preserving functions on $\ab{A}$ finitely generated?
\end{que}
We will now give a brief survey of some known results on this
question: if the algebra is simple, then every finitary function is
congruence preserving and thus the clone of congruence preserving
functions is finitely generated (by its binary members). A slightly
weaker result holds for algebras whose congruences permute with
respect to the relation product and a distributive congruence
lattice; such algebras are called \emph{arithmetical}. It follows
from \cite[Corollary~5.3 (1)]{Ai:OHAH} that the congruence
preserving functions of a finite arithmetical algebra are generated
by their ternary members. Other examples for finite generation are
provided by \emph{affine complete} algebras (cf. \cite{KP:PCIA}): in
a finite affine complete algebra of finite type, every congruence
preserving function is a polynomial function and the clone of
polynomial functions is finitely generated for such algebras.
The results proved in this paper will provide a complete answer to Question~\ref{que:main}
for finite nilpotent groups. For the special case of abelian groups,
we have:
\begin{thm} \label{thm:abelian}
   The clone of congruence preserving functions of a finite abelian group
   is finitely generated if and only if all its Sylow subgroups
   are either cyclic or affine complete.
\end{thm}
The proof of this result will be given in Section~\ref{sec:abelian}.
Hence, using W.\ N\"obauer's description of finite affine complete groups \cite[Satz 5]{No:UDAV},
we gather that $\Z_{27} \times \Z_{25} \times \Z_{25} \times \Z_5$ has a finitely
generated clone of congruence preserving functions, whereas the clone
of congruence preserving functions of $\Z_3 \times \Z_{25}\times \Z_{5}$ is not finitely generated.
The result by Lausch and N\"obauer characterizing affine complete abelian groups can be
stated differently in terms of the subgroup lattice of the group; namely, a finite abelian
$p$-group is affine complete if and only if its subgroup lattice cannot be written as the union of
two proper subintervals. It is therefore not surprising that the shape
of the congruence lattice also plays an important role in describing arbitrary algebras with a finitely
generated clone of congruence preserving functions. In fact, for a finite nilpotent group $\ab{G}$,
we obtain the following description of when the clone of congruence preserving functions
is finitely generated. In this introductory part, we just state the easiest case.
\begin{thm} \label{thm:pgroups1}
   Let $p$ be a prime, and let $\ab{G}$ be a finite $p$-group. We assume that $\{1\}$ and $G$ are the
   only normal subgroups that are comparable (w.r.t. $\subseteq$) to all other normal subgroups of $\ab{G}$.
   Then the following are equivalent:
   \begin{enumerate}
       \item The clone of congruence preserving functions of $\ab{G}$ is not finitely generated.
       \item \label{it:tp2} $\ab{G}$ is not cyclic, and there  exist normal subgroups $E, D$ of $\ab{G}$ such that $E \neq \{1\}$, $D \neq G$, and
every normal subgroup $I$ satisfies $I \ge E$ or $I \le D$.
   \end{enumerate}
\end{thm}
From this criterion, we obtain for example that the congruence
preserving functions of $\Z_9 \times \Z_3$ cannot be finitely
generated (take $E:= 3 \Z_9 \times\{0\}$, $D := 3 \Z_9 \times
\Z_3$), whereas the clone of congruence preserving functions of
$\Z_9 \times \Z_9$ is finitely generated (which is clear because
this group is known to be affine complete); another example to which
Theorem~\ref{thm:abelian} can be applied is the $64$-element group
with number $64/73$ in the catalogue of small groups in GAP; it is a
semidirect product of $\Z_2 \times \Z_2 \times D_8$ with $\Z_2$.
Using GAP \cite{GG:GAP4}, one can find that it does not have any
normal subgroups $D,E$ that satisfy the condition given in
item~\eqref{it:tp2}. Hence this group has a finitely generated clone
of congruence preserving functions. We notice that the affine
completeness status of this group has not been determined in the
literature yet. Starting from this case, we obtain a description for
all finite $p$-groups (Theorem~\ref{thm:pgroups}), and hence for all
finite nilpotent groups (Theorem~\ref{thm:nilpotent}).
One crucial property of a $p$-group $\ab{G}$ that we use is the
following: for two minimal normal subgroups $B$ and $C$ of $G$, $BC$
is isomorphic to $\Z_p \times \Z_p$ and central in $\ab{G}$, and
hence there exists a normal subgroup $D$ of $\ab{G}$ with $\{e\} < D
< B  C$ with $D \neq B$, $D \neq C$. Such a property can be
formulated for arbitrary expansions of groups, leading to 
 a generalization of Theorem~\ref{thm:pgroups} in
Theorem~\ref{teorema}.
A part of these results can be generalized beyond expansions on groups:  we prove
that a finite algebra with a Mal'cev term and simple congruence lattice has its clone
of congruence preserving functions finitely generated if and only if either the algebra is
simple or its congruence lattice is not the union of two proper subintervals
(Theorem~\ref{thm:malcev}).

\section{Results}\label{sec:results}
Following \cite[Definition 1.3]{BS:ACIU}, we consider an algebraic
structure $\ab{A}$ as a pair $(A, F)$, where $A$ is a nonvoid set,
and $F$ is a family of finitary operations on $A$. For universal
algebraic concepts used in this paper, we refer to \cite{BS:ACIU};
some of the concepts that are most relevant to this paper will be
introduced below. We use $\N$ for the set of nonzero natural
numbers, and $\N_0$ for $\N \cup \{0\}$.

\begin{de}
   Let $\ab{A}$ be an algebra, let $k \in \N_0$, and let $f: A^k \to A$. We say that
   $f$ \emph{preserves} a binary relation $\alpha$ on $A$ if for all
   $a_1,\ldots, a_k, b_1,\ldots, b_k \in A$ with $(a_1,b_1) \in \alpha$, \ldots,
   $(a_k, b_k) \in \alpha$, we have
   $(f (a_1,\ldots, a_k), f(b_1,\ldots, b_k)) \in \alpha$, and we abbreviate this
   fact by $f\rhd\alpha$. A \emph{congruence relation} of the algebra $\ab{A} =
\algop{A}{F}$ is an equivalence relation $\alpha$ on $A$ preserved
by all $f\in F$. 
The set of all congruences of
$\ab{A}$ is denoted by $\Con (\ab{A})$. The function $f$ is
\emph{congruence preserving} if $f\rhd\alpha$ for all $\alpha\in\Con
(\ab{A})$.
   The set of all $k$-ary congruence preserving (or \emph{compatible}, for short) functions is denoted
   by $\C_k (\ab{A})$, and $\C (\ab{A}) := \bigcup_{k \in \N} \C_k (\ab{A})$.
\end{de}

The set $\C (\ab{A})$ is a clone (cf. \cite[p.18]{PK:FUR},
\cite[Definition~4.1]{Be:UA}) on $A$, and we will investigate
whether it is finitely generated. Here we say that  $\C (\ab{A})$ is
finitely generated if there is a finite subset $F$ of $\C (\ab{A})$
such that every function in $\C (\ab{A})$ can be written as a
composition of projections and functions in $F$; a precise
definition of finitely generated clones is given in
\cite[p.50]{PK:FUR}. On a group $\ab{G}$, a function $f:G^k \to G$ is
compatible if and only if $f (g_1, \ldots, g_k)^{-1} \cdot
f(h_1,\ldots, h_k)$ lies in the normal subgroup generated by
$\{g_1^{-1} \cdot h_1, \ldots, g_k^{-1} \cdot h_k \}$ for all $(g_1,
\ldots, g_k),(h_1,\ldots, h_k)\in G^k$.

Our description of algebras with finitely generated clone of
compatible functions uses the shape of the congruence lattice of the
algebra. We will therefore need some notions for lattices (cf.
\cite{Gr:GLTS}). For a lattice $\ob{L}$ and $a,b \in \ob{L}$ with
$a\leq b$, we define the  \emph{interval} $\IntI [a,b]$ by $\IntI
[a,b]:= \{c \in \ob{L}\mid a\leq c \leq b\}$.
We denote the four element Boolean lattice by $\ob{M}_2$, and we 
say that a lattice $\ob{L}$ is \emph{$\ob{M}_2$-free} if it has no
interval that is isomorphic to $\ob{M}_2$. 
 We say that an element
$\beta$ of a bounded lattice $\ob{\ob{L}}$ \emph{cuts} the lattice
if  for all $\alpha \in \ob{L}$, we have $\alpha \le \beta$ or
$\alpha \ge \beta$; this is equivalent to saying that $\ob{L}$ is
the union of the intervals $\IntI [{\mathbf 0},\beta]$ and $\IntI
[\beta, {\mathbf 1}]$. Let $C$ be the set of cutting elements of
$\ob{L}$. Then $C$ is obviously a linearly ordered subset of
$\ob{L}$ that contains ${\mathbf 0}$ and ${\mathbf 1}$. For a
lattice $\ob{L}$ of finite height, let ${\mathbf 0} = \gamma_0 <
\cdots < \gamma_k = {\mathbf 1}$ be the sequence of all cutting
elements of $\ob{L}$. Then it is easy to see that $\ob{L} = \IntI
[\gamma_0, \gamma_1] \cup \IntI [\gamma_1, \gamma_2] \cup \cdots
\cup  \IntI [\gamma_{k-1}, \gamma_k]$. If $\beta$ cuts the lattice,
we say that $\ob{\ob{L}}$ is a \emph{coalesced ordered sum} of the
intervals $\IntI [{\mathbf 0}, \beta]$ and $\IntI [\beta, {\mathbf
1}]$.

A special role will be played by those lattices that can be written
as a union of two proper subintervals; we say that such lattices
\emph{split}. For a lattice $\ob{L}$, we say that a pair of elements
$(\delta, \eps) \in \ob{L}^2$ \emph{splits} $\ob{L}$ if $\delta <
{\mathbf 1}$, $\eps > {\mathbf 0}$, and for all $\alpha \in \ob{L}$,
we have $\alpha \le \delta$ or $\alpha \ge \eps$. Note that if
$(\delta, \eps)$ splits the lattice $\ob{L}$, then $\mathbb{L} =
\IntI [{\mathbf 0},\delta] \cup \IntI [\eps , {\mathbf 1}]$. The
lattice $\ob{L}$ \emph{splits} if it has a splitting pair. This
property has already been used in \cite{QW:SROC, AM:SOCO}. For
example, the lattices $\ob{M}_2$ and the congruence lattice of
$\mathbb{Z}_4\times \mathbb{Z}_2$ both split,  and $\ob{M}_3$ does
not split. \vspace{5mm}

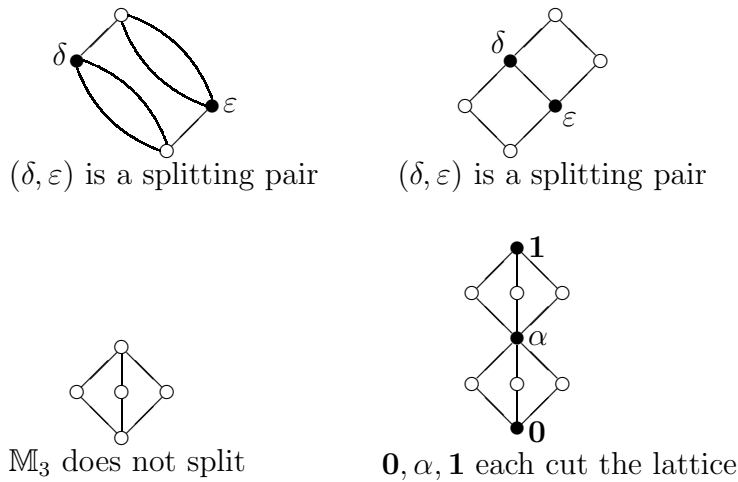
\begin{center}
\setlength{\unitlength}{0.06mm}
\begin{picture}(300,300)(0,-65)
\put(100,300){\circle{30}}
\put(0,200){\circle*{30}}
\put(300,100){\circle*{30}} 
\put(200,0){\circle{30}}
\put(89,289){\line(-1,-1){78}}
\put(289,89){\line(-1,-1){78}}
\qbezier(0,185)(50,50)(185,0)
\qbezier(15,200)(150,150)(200,15)
\qbezier(100,285)(150,150)(285,100)
\qbezier(115,300)(250,250)(300,115)
\put(-20,190){\makebox(0,0)[br]{$\delta$}}
\put(325,116){\makebox(0,0)[tl]{$\epsilon$}} 
\put(-150,-75){($\delta,\epsilon$) is a splitting pair}
\end{picture}
\hspace{3cm}
\begin{picture}(330,400)(-15,-65)
\put(100,0){\circle{30}}
\put(0,100){\circle{30}}
\put(200,100){\circle*{30}}
\put(100,200){\circle*{30}}

\put(200,300){\circle{30}}
\put(300,200){\circle{30}}
\put(211,111){\line(1,1){78}}
\put(111,211){\line(1,1){78}}
\put(289,211){\line(-1,1){78}}

\put(111,11){\line(1,1){78}}
\put(11,111){\line(1,1){78}}
\put(89,11){\line(-1,1){78}}
\put(189,111){\line(-1,1){78}}
\put(90,216){\makebox(0,0)[br]{$\delta$}}
\put(215,85){\makebox(0,0)[tl]{$\epsilon$}}
\put(-150,-75){$(\delta,\varepsilon)$ is a splitting pair}
\end{picture}
\end{center}
\vspace{7mm}
\begin{center}
\setlength{\unitlength}{0.06mm}
\begin{picture}(330,400)(-15,-65)
\put(100,0){\circle{30}}
\put(0,100){\circle{30}}
 \put(100,100){\circle{30}}
\put(200,100){\circle{30}}
\put(100,200){\circle{30}}

\put(111,11){\line(1,1){78}}
\put(11,111){\line(1,1){78}}

\put(89,11){\line(-1,1){78}}
\put(189,111){\line(-1,1){78}}
\put(100,16){\line(0,1){68}}
\put(100,116){\line(0,1){68}}
\put(-150,-75){$\ob{M}_3$ does not split}
\end{picture}
\hspace{30mm}
\begin{picture}(330,400)(-15,-85)
\put(100,0){\circle*{30}}
\put(0,100){\circle{30}}
 \put(100,100){\circle{30}}
\put(200,100){\circle{30}}
\put(100,200){\circle*{30}}
\put(0,300){\circle{30}}
\put(100,300){\circle{30}}
\put(200,300){\circle{30}}
\put(100,400){\circle*{30}}
\put(89,211){\line(-1,1){78}}
\put(111,11){\line(1,1){78}}
\put(11,111){\line(1,1){78}}
\put(111,211){\line(1,1){78}}
\put(11,311){\line(1,1){78}}
\put(89,11){\line(-1,1){78}}
\put(189,111){\line(-1,1){78}}

\put(189,311){\line(-1,1){78}}
 \put(100,16){\line(0,1){68}}
 \put(100,116){\line(0,1){68}}
\put(100,216){\line(0,1){68}} \put(100,316){\line(0,1){68}}
\put(125,15){\makebox(0,0)[tl]{$\mathbf 0$}}
 \put(125,215){\makebox(0,0)[tl]{$\alpha$}}
 \put(125,425){\makebox(0,0)[tl]{$\mathbf 1$}}
 \put(-200,-100){${\mathbf 0},\alpha,{\mathbf 1}$ each cut the lattice}
\end{picture}
\end{center}
\vspace{2mm}
\begin{thm} \label{thm:pgroups}
    Let $\ab{G}$ be a finite $p$-group, let $\ob{L}$ be the lattice of normal subgroups
    of $\ab{G}$, and let $\{e\} = N_0 < \cdots < N_n = G$ be the sequence of those normal subgroups
    that cut the lattice $\ob{L}$. Then the following are equivalent:
    \begin{enumerate}
     \item The clone of congruence preserving functions of $\ab{G}$ is finitely generated.
     \item For each $i \in \{0, \ldots, n-1\}$, the interval $\IntI [N_i, N_{i+1}]$ of the lattice
           of normal subgroups of $\ab{G}$  either
           contains exactly two elements, or $\IntI [N_i, N_{i+1}]$ does not split.
    \end{enumerate}
\end{thm}
The proof is given in Section \ref{sec:mainproof}.
As a consequence of this Theorem, the clone of congruence preserving functions of
the eight element dihedral group is finitely generated.
From this result, it is easy to determine the finite generation of congruence preserving functions
for all finite nilpotent groups, since we have:
\begin{thm} \label{thm:nilpotent}
    The clone of congruence preserving functions of a finite nilpotent group is
    finitely generated if and only if the clone of congruence preserving
    functions of every Sylow subgroup is finitely generated.
\end{thm}
The proof follows from the more general result
Proposition~\ref{pro:dp2} below because every finite nilpotent
group is isomorphic to a skew-free direct product of its Sylow
subgroups. The result for $p$-groups can be generalized to certain
expansions of groups. We call an algebra $\A$ an \emph{expanded
group} if it has operations $+$ (binary), $-$ (unary) and $0$
(nullary) and its reduct $(A,+,-,0)$ is a (not necessarily abelian)
group. In expanded groups, congruences are described by ideals
\cite{Ku:LOGA}. A subset $I$ of $A$ is an \emph{ideal} of $\A$ if
$I$ is a normal subgroup of $\algop{A}{+}$, and for all $k\in
\mathbb{N}$, for all $k$-ary fundamental operations $f$ of $\A$, for
all $(a_1,\dots,a_k)\in A^k$ and $(i_1,\dots,i_k)\in I^k,$ we have
$$f(a_1+i_1,\dots,a_k+i_k)-f(a_1,\dots,a_k)\in I.$$ The set of all
ideals of $\A$ is denoted by $\mathrm{Id}(\A)$, and the lattice
$\algop{\mathrm{Id} (\A)}{+, \cap}$ is isomorphic to the congruence
lattice of $\A$ via the correspondence $\gamma : \Con(\ab{A}) \to
 \mathrm{Id} (\ab{A})$, $\alpha \mapsto 0/\alpha$.
We can
generalize Theorem~\ref{thm:pgroups} to those expanded groups whose
ideal lattice has no interval isomorphic to $\ob{M}_2$.
This is really a generalization because the normal subgroup lattice
of a $p$-group cannot have an interval isomorphic to $\ob{M}_2$. Another
class of expanded groups with $\ob{M}_2$-free ideal lattices is the
class of finite local commutative rings.
\begin{thm}\label{teorema}
    Let $\ab{A}$ be a finite expanded group with $\ob{M}_2$-free ideal lattice, and
    let  $\{0\} = S_0 < \cdots < S_n = A$ be the set of those ideals
    that cut the lattice $\mathrm{Id} (\ab{A})$. Then the following are equivalent:
    \begin{enumerate}
     \item \label{it:t1} The clone of congruence preserving functions of $\ab{A}$ is finitely generated.
     \item \label{it:t2} For each $i \in \{0, \ldots, n-1\}$, the interval $\IntI [S_i, S_{i+1}]$ of the lattice
           $\mathrm{Id} (\ab{A})$  either
           contains exactly the two elements $S_i$ and $S_{i+1}$,
           or $\IntI [S_i, S_{i+1}]$ does not split.
    \end{enumerate}
\end{thm}
The proof of this theorem is given
in Section
\ref{sec:mainproof}. Moreover, we are going to prove that the implication
$(2)\Rightarrow (1)$ in this theorem is also true without the assumption that
the ideal lattice of $\ab{A}$ is  $\ob{M}_2$-free.
 More precisely, in Section~\ref{sec:mainproof}, we will prove the
following theorem.

\begin{thm}\label{thm:oppositewithoutM2}
Let $\ab{A}$ be an expanded group and let $\{0\} = S_0 < \cdots <
S_n = A$ be the set of those ideals that cut the lattice
$\mathrm{Id} (\ab{A})$. Suppose that for each $i \in \{0, \ldots,
n-1\}$, the interval $\IntI [S_i, S_{i+1}]$ of the lattice
$\mathrm{Id} (\ab{A})$ either contains exactly the two elements
$S_i$ and $S_{i+1}$, or the interval $\IntI [S_i, S_{i+1}]$ does not
split.
Then the clone of congruence
preserving functions of $\ab{A}$ is finitely generated.
\end{thm}

 If the congruence lattice of an algebra is of
a special shape, namely, if it is a simple lattice, an instance of
Theorem \ref{teorema} generalizes to arbitrary algebras with Mal'cev
term.
\begin{thr} \label{thm:malcev}
Let $\mathbf{A}$ be a finite algebra with a Mal'cev term,
and let $\mathbb{L}$ be the congruence lattice of $\mathbf{A}$.
If $\mathbb{L}$ is simple, then the following are equivalent:
\begin{enumerate}
\item $\mathrm{Comp}(\mathbf{A})$ is finitely generated;
\item $\mathbb{L}$ does not split, or $|\ob{L}| \le 2$.
\end{enumerate}
\end{thr}
The proof is given in Section \ref{Malcev}.

\section{Preliminaries from universal algebra}

We will first state some elementary facts on congruence preserving
functions. For an algebra $\ab{A}$, we denote its congruence lattice
by $\Con (\ab{A})$. If $n \in \N$, $f \in \Comp_n (\ab{A})$ and
$\alpha \in \Con (\A)$, then $f^{\A / \alpha}$ given by
   \(
       f^{\A / \alpha} (a_1 / \alpha, \ldots, a_n / \alpha) :=
          f (a_1, \ldots, a_n)  / \alpha
   \) for all $a_1,\ldots, a_n \in A$
    is well-defined and belongs to $\Comp_{n} (\A / \alpha)$.
An important construction of congruence preserving functions
comes from splitting pairs in the congruence lattice.
For an $n$-tuple $\vb{a} = (a_1, \ldots, a_n) \in A^n$ and
a congruence $\alpha \in \Con (\ab{A})$, we let
$\vb{a} / \alpha := \{ (b_1, \ldots, b_n) \in A^n \setsuchthat
                       (b_1, a_1) \in \alpha, \ldots, (b_n, a_n) \in \alpha \}$.
\begin{pro} \label{pro:de}
    Let $\A$ be an algebra, let $n \in \N$, and let $(\delta, \eps)$ be
    a splitting pair of $\Con (\ab{A})$.
    Let $a \in A$, and  let $f : A^n \to A$ be a function with $f(A^n) \subseteq a / \eps$.
    We assume
    that for every $\vb{b} \in A^n$ there exists a function $g_{\vb{b}} \in \Comp_n (\ab{A})$
    such that for all $\vb{x} \in \vb{b}/\delta$, we have
    $f(\vb{x}) = g_{\vb{b}} (\vb{x})$.
     Then $f \in \Comp (\ab{A})$.
\end{pro}
\emph{Proof:}
     Let $\vb{x}, \vb{y} \in A^n$, and let $\varphi \in \Con (\A)$ be such that
     $(x_i, y_i) \in \varphi$ for all $i \in \{1,\ldots, n\}$.
     We show $(f(\vb{x}), f (\vb{y})) \in \varphi$.
          The first case is that $(x_i, y_i) \in \delta$ for all $i \in \{1,\ldots, n\}$.
     Then there is a $\vb{b} \in A^n$ such that $\vb{x}/\delta = \vb{b}/\delta$ and
     $\vb{y} / \delta = \vb{b} / \delta$. Let $g_{\vb{b}}$ be the congruence preserving
     function that interpolates $f$ on $\vb{b} / \delta$. Since $g_{\vb{b}}$ is congruence
     preserving, $(g_{\vb{b}} (\vb{x}), g_{\vb{b}} (\vb{y})) \in \varphi$, and therefore
     $(f(\vb{x}), f(\vb{y})) \in \varphi$.
      The second case is that there is an $i \in \{1,\ldots, n\}$ with
     $(x_i, y_i) \not\in \delta$. Then $\varphi \not\le \delta$, and
     therefore $\varphi \ge \eps$. Since $f (A^n) \subseteq a / \eps$,
      we have $(f(\vb{x}), f (\vb{y})) \in \eps$, and thus $(f (\vb{x}), f(\vb{y})) \in \varphi$.
     \qed

     This property will be particularly useful in the case that $\alpha$ cuts
     the congruence lattice of $\ab{A}$. In this case, either $\alpha \in \{{\mathbf 0}, \mathbf{1}\}$,
     or we may use Proposition~\ref{pro:de} with $\delta = \eps = \alpha$.

We will now investigate how congruence preserving functions act on
direct products. Let $\ab{A}, \ab{B}$ be similar algebras, and let
$n \in \N$. For $\vb{a} \in A^n$ and $\vb{b} \in B^n$, let $(\vb{a},
\vb{b})^T$ denote the tuple $ ( (a_1, b_1), \ldots, (a_n, b_n)) \in
(A \times B)^n$. For $\alpha \in \Con (\ab{A})$ and $\beta \in \Con
(\ab{B})$, we write $\alpha \times \beta$ for the congruence of
$\ab{A} \times \ab{B}$ given by
\[
   \alpha \times \beta = \{ ((a_1, b_1), (a_2, b_2)) \setsuchthat (a_1, a_2) \in \alpha, (b_1, b_2) \in \beta \}.
\]
\begin{pro}[{\cite[Lemma~4]{No:UDAV}}]  \label{pro:hfg}
   Let $\ab{A}$ and $\ab{B}$ be similar algebras, and let $h \in \Comp_n (\ab{A}
   \times \ab{B})$. Then there are $f \in \Comp_n (\ab{A})$ and $g \in \Comp_n (\ab{B})$
   such that
   \begin{equation} \label{eq:hab}
         h ( (\vb{a}, \vb{b})^T ) = (f (\vb{a}), g (\vb{b}))
   \end{equation}
   for all $\vb{a} \in A^n$, $\vb{b} \in B^n$.
\end{pro}
  We note that given $h$,
   the pair $(f,g)$ is uniquely determined by~\eqref{eq:hab}.
We will denote the functions $f$ and $g$ that satisfy \eqref{eq:hab} also by
$\phi_{\ab{A}} (h)$ and $\phi_{\ab{B}} (h)$.
A direct product of two similar algebras $\ab{A} \times \ab{B}$ is called \emph{skew-free}
\cite{BS:ACIU},
if for every congruence $\gamma \in \Con (\ab{A} \times \ab{B})$, there are congruences
$\alpha \in \Con (\ab{A})$ and $\beta \in \Con (\ab{B})$ such that
$\gamma = \alpha \times \beta$.
\begin{pro}[cf. {\cite[Satz~1]{No:UDAV}}] \label{pro:dp1}
   Let $n \in \N$, and  $\ab{A}, \ab{B}$ be two similar algebras.
   We assume in addition that the direct
   product $\ab{A} \times \ab{B}$ is skew-free. Then we have:
   \begin{enumerate}
      \item \label{it:d1} For every $f \in \Comp_n (\ab{A})$ and $g \in \Comp_{n} (\ab{B})$,
            the function $h : (A \times B)^n \to A \times B$,
            $h ( (\vb{a}, \vb{b})^T ) :=
            (f (\vb{a}), g (\vb{b}))$ satisfies $h \in \Comp_n (\ab{A} \times \ab{B})$.
      \item \label{it:d2} The function $e : (A \times B)^2 \to A \times B$, $e ( (a_1, b_1), (a_2, b_2) ) :=
            (a_1, b_2)$ is a congruence preserving function of $\ab{A} \times \ab{B}$.
    \end{enumerate}
\end{pro}
\emph{Proof:}
    Item  \eqref{it:d1} is Satz~1 from \cite{No:UDAV}.
    For item~\eqref{it:d2}, we apply~\eqref{it:d1} for $f(x,y) := x$ and $g(x,y) := y$. \qed

Given an algebra $\A$ and $k\in\mathbb{N}$, a function $f:A^k\to A$ is called
a \emph{$k$-ary term function} if there is a $k$-ary term $t$ in the language of $\A$ such that
$f(x_1, \ldots, x_k) = t^{\ab{A}} (x_1,\ldots, x_k)$ for all $x_1,\ldots, x_k \in A$.
The set of term functions on $\A$ will be denoted by $\Clo (\A)$.
 The function $f$ is a
 \emph{$k$-ary polynomial function}
(or \emph{polynomial}) of $\A$ if there are a natural number $l$, elements $a_1,\dots,a_l\in A$, and a $(k+l)$-ary term
$t$ in the language of $\A$ such that
$$f(x_1,\dots,x_k)=t^{\A}(x_1,\dots,x_k,a_1,\dots,a_l)$$
for all $x_1,\dots,x_k\in A$. By $\mathrm{Pol}_k(\A)$ we denote the set of all $k$-ary polynomials of $\A$, and
$\mathrm{Pol}(\A):=\bigcup_{k\in\N}\mathrm{Pol}_k(\A)$. Note that $\mathrm{Pol}(\A)$ is a clone for each algebra $\A$
and we call it the \emph{clone of polynomial functions} of $\A$.

\begin{df}
 Let $\A$ be an algebra, let $k\in\mathbb{N}$, let $p:A^k\rightarrow A$, let $(a_1,\dots, a_{k})\in A^k,$ and let $o\in A$.
Then $p$ is \emph{absorbing at $(a_1,\dots, a_{k})$ with value $o$} if for all $(x_1,\dots,x_{k})\in A^k$ we have:
if there is an $i\in \{1,\dots,k\}$ with $x_i=a_i$, then $p(x_1,\dots,x_{k}) = o$.

\end{df}

A ternary operation $m$ on a set $A$ is said to be a \emph{Mal'cev operation}
if we have $m(x,x,y)=y=m(y,x,x)$  for all $x,y\in A$.
An algebra $\A$ is called a \emph{Mal'cev algebra}
if $\A$ has a Mal'cev operation among its ternary term functions.
For congruences $\alpha$ and $\beta$ of $\A$, we will denote their
\emph{binary commutator} by $[\alpha,\beta]_{\A}$ (see
\cite{FM:CTFC} and \cite[p.252]{MMT:ALVV}). We are going to omit the
index that denotes the algebra whenever it is clear from the
context.
 Here are some properties of binary commutators for congruence modular varieties
that we are going to use. For all $\alpha,\beta\in\mathrm{Con}(\A)$:
\begin{itemize}
\item[(BC1)] $[\alpha,\beta]\leq \alpha\land\beta;$
\item[(BC2)] for all $\gamma,\delta\in \mathrm{Con}(\A)$ such that $\alpha\leq \gamma, \beta\leq\delta,$ we have $[\alpha,\beta]\leq [\gamma,\delta]$;
\item[(BC7)] if $\A$ generates a congruence permutable variety, then for a nonempty set $I$ and $\{\rho_i\mid i\in I\}\subseteq \mathrm{Con}(\A)$ we have that
    $$\bigvee_{i\in I}[\alpha,\rho_i]=[\alpha,\bigvee_{i\in I}\rho_i] \,\;\mbox{ and }\,\; \bigvee_{i\in I}[\rho_i,\beta]=[\bigvee_{i\in I}\rho_i,\beta].$$
\end{itemize}
\emph{Higher commutators} have been introduced in \cite{Bu:OTNO};
their properties in congruence permutable varieties have been
investigated in \cite{AM:SAOH}. For $k \in \N$, an algebra is called
\emph{$k$-supernilpotent} if $[\underbrace{{\mathbf 1},\dots
,{\mathbf 1}}_{k+1}]={\mathbf 0}$. An algebra $\mathbf{A}$ is called
\emph{supernilpotent} if there exists a $k\in \mathbb{N}$ such that
$\mathbf{A}$ is $k$-supernilpotent. Following \cite[p.58]{FM:CTFC}
we call an algebra $\A$ \emph{nilpotent} if in the chain of
congruences of $\A$ defined by \(
   \gamma_1:=[{\mathbf 1},{\mathbf 1}], \text{ and }
   \gamma_i:=[{\mathbf 1},\gamma_{i-1}] \text{ for every } i>1,
\) there is an $n\in \mathbb{N}$ such that $\gamma_n={\mathbf 0}$.

The following results concerning supernilpotent algebras will be used in the proof of the Theorem \ref{thm:malcev}.

\begin{pr}(cf. \cite[Theorem 3.14]{Ke:CMVW})\label{Lema7.6.AM:SAOH}
Let $\A$ be a finite nilpotent algebra of finite type that generates a congruence modular variety. Then
$\A$ factors as a direct product of algebras of prime power cardinality if and only if
$\A$ is a supernilpotent Mal'cev algebra.
\end{pr}

A function $f:A^n \to A$ \emph{depends} on its $i$th argument if there
are $a_1, \ldots, a_n, b_i \in A$ such that
$f(a_1, \ldots, a_n) \neq f(a_1, \ldots, a_{i-1}, b_i, a_{i+1},\ldots, a_n)$.
The \emph{essential arity} of $f$ is the number of arguments on which 
$f$ depends.
\begin{pr}(cf. \cite[Corollary 6.17]{AM:SAOH})\label{Cor6.17.AM:SAOH}
Let $k\in\N$. In a $k$-supernilpotent Mal'cev algebra, every absorbing polynomial has essential arity at most $k$.
\end{pr}

\begin{pr}(cf. \cite[Lemma 3.3]{AM:SOCO})\label{Lema3.3.AM:SOCO}
 Every finite algebra whose congruence lattice does not split is supernilpotent.
\end{pr}

Following C.\ Bergman \cite{Be:UA}, the clone generated by a set of functions $F$ on a set
$A$ will be denoted by $\Clo^A (F)$ or $\Clo (F)$, and we write
$ar(f)$ for the  arity of an operation $f$.

\begin{pr}\cite[Proposition 6.18]{AM:SAOH}\label{pr6.18.AM:SAOH}
Let $k\in\N$. If $\A$ is a $k$-supernilpotent Mal'cev algebra, with a Mal'cev term $m$, then
$\Clo^A (\Pol_k (\ab{A} )\cup \{m\} ) = \mathrm{Pol}(\mathbf{A})$.
\end{pr}

In investigating clone generation,
we find it useful to use the approach to clones that
goes back to A.I.Mal'cev, and is given in
\cite[p.38]{PK:FUR}. We recall that for a set $A$, on the set
$P_A=\bigcup_{n \in \N} A^{A^n}$, one defines the operations
$\zeta,\tau,\Delta,\nabla, \circ$ such that for $f,g\in P_A$,
$ar(f)=n, ar(g)=m$, we have that $ar(\zeta f)=n, ar(\tau f)=n,
ar(\Delta f)=n-1, ar(g\circ f)=n+m-1$, $ar(\nabla f)=n+1$, and
\begin{enumerate}
\item $(\zeta f)(x_1,x_2,\dots, x_n):= f(x_2,\dots, x_n, x_1),$
\item $(\tau f)(x_1,x_2,x_3,\dots, x_n):= f(x_2,x_1,x_3,\dots, x_n),$
\item  $(\Delta f)(x_1,x_2,\dots, x_{n-1}):= f(x_1, x_1,x_2,\dots, x_{n-1}),$
\end{enumerate}
for $n\geq 2$, and $\zeta f = \tau f = \Delta f = f$ for $n=1$,
\begin{enumerate}
\setcounter{enumi}{3}
\item $(\nabla f)(x_1,x_2,\dots, x_{n+1}):= f(x_2,\dots, x_{n+1}),$
\item $(g\circ f)(x_1,\dots, x_{m+n-1}):= f(g(x_1,\dots, x_m),x_{m+1},\dots, x_{m+n-1})$
\end{enumerate}
for all $x_1,\dots,x_{m+n-1}\in A.$
The function $\id_A \in A^{A}$ is the identity function on $A$.
Following \cite{PK:FUR}, we call
$\algop{P_A}{\id_A, \zeta, \tau, \Delta, \nabla, \circ}$ the full \emph{function algebra}
on $A$. The subuniverses of the full function algebra are exactly the clones on $A$.
For a nonempty set $F$ of finitary operations on $A$,
the \emph{function algebra generated by $F$}
is the subalgebra of $\algop{P_A}{\id_A, \zeta, \tau, \Delta, \nabla, \circ}$
generated by $F$, and its universe is $\Clo^A (F)$.
 This approach allows to use the concepts of classical universal algebra
 (cf. \cite[Bemerkungen~1.1.3~(v)]{PK:FUR}) for the generation of clones:
\begin{lem} \label{lem:clohom}
    Let $A,B$ be sets, let $C$ be a clone on $A$, let $D$ be a clone on $B$,
    let $\varphi$ be a homomorphism from the function algebra
    $\algop{C}{\id_A, \zeta, \tau, \Delta, \nabla, \circ}$ into
    $\algop{D}{\id_B, \zeta, \tau, \Delta, \nabla, \circ}$, and let
    $f_1,\ldots, f_n$ be finitary operations on $A$.
    Then we have:
    \begin{enumerate}
        \item
         If
         $\Clo^A (f_1,\ldots, f_n) = C$, then
         $\Clo^B (\varphi (f_1), \ldots, \varphi (f_n)) =
           \varphi(C)$.
       \item If $\Clo^B (\varphi (f_1), \ldots, \varphi (f_n)) = D$,
          then $\varphi (\Clo^A (f_1,\ldots, f_n)) = D$.
    \end{enumerate}
\end{lem}
 \emph{Proof:} Both properties are consequences of~\cite[Theorem~II.6.6]{BS:ACIU}. \qed

\section{Generation of congruence preserving functions}

For similar algebras $\ab{A}$, $\ab{B}$ and a homomorphism $h : \ab{A} \to \ab{B}$,
the connections between the congruence preserving functions of $\ab{A}$, $\ab{B}$,
$\ab{A} \times \ab{B}$, and $h (\ab{A})$ are less obvious than one might wish.
Some of these connections are collected in this section.
\begin{pro} \label{pro:dp2}
   Let $n \in \N$, and let $\ab{A}$ and  $\ab{B}$ be similar algebras. We assume that the direct
   product $\ab{A} \times \ab{B}$ is skew-free.
   Then the clone $\Comp (\ab{A} \times \ab{B})$ is finitely generated if and only
            if both clones $\Comp (\ab{A})$ and $\Comp (\ab{B})$ are finitely generated.
\end{pro}
  \emph{Proof:}
     For the ``if''-direction, let $\Comp (\ab{A}) =
    \Clo^A (f_1,\ldots, f_k)$ and $\Comp (\ab{B})$  $=$  $\Clo^B (g_1, \ldots, g_l)$,
     and let $e$ be the congruence preserving function produced in Proposition~\ref{pro:dp1}~\eqref{it:d2}.
     For $n \in \N$ and a function $f \in \Comp_n (\ab{A})$,
     let $f' : (A \times B)^n \to A \times B$ be defined by
         $f' ( (\vb{a}, \vb{b})^T ) := (f (\vb{a}), b_1)$ for all $\vb{a} = (a_1,\ldots, a_n) \in A^n$ and $\vb{b} = (b_1,\ldots,b_n)
      \in B^n$ ; similarly,
     for $g \in \Comp_n (\ab{B})$, let
       $g'' ( (\vb{a}, \vb{b})^T ) := (a_1, g (\vb{b}))$.
     The functions $f'$ and $g''$ lie in $\Comp (\ab{A} \times \ab{B})$ because of Proposition~\ref{pro:dp1}~\eqref{it:d1}.
     We claim that $\{ f_1', \ldots, f_k', g_1'', \ldots, g_l'', e \}$ generates
     $\Comp ( \ab{A} \times \ab{B})$.
     To prove this, we observe that
     the mapping
     $\phi_{\ab{A}} : \Comp (\ab{A} \times \ab{B}) \to \Comp ( \ab{A} )$
     that was introduced after the proof of Proposition~\ref{pro:hfg}
     is a homomorphism from the function
     algebra $(\Comp (\ab{A} \times \ab{B}), \id_{A \times B}, \zeta, \tau, \Delta, \nabla, \circ)$
       into  $(\Comp (\ab{A}), \id_{A}, \zeta, \tau, \Delta, \nabla, \circ)$.
     Similarly, $\phi_{\ab{B}}$ maps
      $\Comp (\ab{A} \times \ab{B})$ into $\Comp (\ab{B})$.
     Now let $h \in \Comp_n (\ab{A} \times \ab{B})$. By Proposition~\ref{pro:dp1},
      $\phi_{\ab{A}} (h) \in \Comp_n (\ab{A})$ and $\phi_{\ab{B}} (h) \in \Comp_n (\ab{B})$.
      Since $\phi_{\ab{A}} (f_i') = f_i$, the function $\phi_{\ab{A}} (h)$ lies in the clone on $\ab{A}$ generated
      by $\{\phi_{\ab{A}} (f_1'), \ldots, \phi_{\ab{A}} (f_k') \}$. By Lemma~\ref{lem:clohom}, there is an
      $F \in
      \Clo^{A \times B} ( \{f_1', \ldots, f_k' \} )$ such that $\phi_{\ab{A}} (F) = \phi_{\ab{A}} (h)$.
      Similarly, there is $G \in \Clo^{A \times B} ( \{g_1', \ldots, g_l' \} )$ such that $\phi_{\ab{B}} (G) =
      \phi_{\ab{B}} (h)$.
            Then $e (F, G) \in
        \Clo^{A \times B} ( \{ f_1', \ldots, f_k', g_1'', \ldots, g_l'', e \} )$.
       We will now show $e(F,G) = h$. To this end, let $\vb{a} \in A^n$ and $\vb{b} \in B^n$.
       Then there are $x \in A$ and $y \in B$ such that
      $e(F ( (\vb{a}, \vb{b})^T), G ( (\vb{a}, \vb{b})^T)) =
       e( (\phi_{\ab{A}} (h) (\vb{a}), y), (x, \phi_{\ab{B}} (h) (\vb{b}))) =
          (\phi_{\ab{A}} (h) (\vb{a}), \phi_{\ab{B}} (h) (\vb{b})) = h ( (\vb{a}, \vb{b})^T)$.
       Thus $h \in \Clo^{A \times B} ( \{ f_1', \ldots, f_k', g_1'', \ldots, g_l'', e \} )$.

       For the ``only if''-direction, we assume that $\Comp (\ab{A} \times \ab{B})$ is finitely
       generated. By Proposition~\ref{pro:dp1}~\eqref{it:d1}, the mapping   $\phi_{\ab{A}}$ is surjective,
       and thus it is a function algebra epimorphism
       from $(\Comp (\ab{A} \times \ab{B}), \id_{A \times B}, \zeta, \tau, \Delta, \nabla, \circ)$ onto
       $(\Comp (\ab{A}), \id_A, \zeta, \tau, \Delta, \nabla, \circ)$. Since homomorphic
        images of finitely generated algebras are finitely generated, it follows
       that $\Comp (\ab{A})$ is finitely generated. Similarly, $\Comp (\ab{B})$ is
       finitely generated. \qed

     In Section~\ref{sec:abelian} we will examine abelian groups and
       see  that
     $\Comp (\Z_2)$, $\Comp (\Z_4)$ and $\Comp (\Z_4 \times \Z_4 \times \Z_2 \times \Z_2)$
     are all finitely generated, whereas $\Comp (\Z_4 \times \Z_2)$ is not finitely
     generated. This shows that none of the implications of Proposition~\ref{pro:dp2}
     holds if the assumption that the product is skew-free is omitted.

   Next, we will see how finite generation can be preserved under forming certain
   homomorphic images.
    \begin{pro} \label{pro:lift1}
        Let $\ab{A}$ be an algebra, and let $\alpha \in \Con (\ab{A})$ such that
        for every $f \in \Comp (\ab{A} / \alpha)$, there exists an
        $\tilde{f} \in \Comp (\ab{A})$ with $\tilde{f}^{\ab{A} / \alpha} = f$.
        Then if $\Comp (\ab{A})$ is finitely generated, then so is
        $\Comp (\ab{A}/\alpha)$.
    \end{pro}
    \emph{Proof:} The mapping $\psi: \Comp (\ab{A}) \to \Comp (\ab{A} / \alpha)$,
         $g \mapsto g^{\ab{A}/\alpha}$, is a function algebra homomorphism. By the assumptions,
         it is surjective, and thus $\Comp (\ab{A} / \alpha)$ is a homomorphic
        image of $\Comp (\ab{A})$. \qed
     \begin{pro} \label{pro:cut1}
          Let $\ab{A}$ be an algebra, let $\alpha$ be a congruence
          of $\ab{A}$ that cuts $\Con (\ab{A})$, and let $f \in \Comp (\ab{A}/\alpha)$.
          Let $R$ be a set of representatives of $A$ modulo $\alpha$,
         and for $a \in A$, let $r (a)$ be the unique element of $R$ with
         $(a, r(a)) \in \alpha$.
         Let $\tilde{f}:A^{ar(f)} \to A$ be defined by
         $\tilde{f} (\vb{x}) := r ( f (\vb{x} / \alpha ))$.
         Then $\tilde{f} \in \Comp (\ab{A})$.
     \end{pro}
     \emph{Proof:}
         Let $\vb{x}, \vb{y} \in A^n$, and let $\beta \in \Con (\ab{A})$
         with $\vb{x} \equiv_{\beta} \vb{y}$. Since $\alpha$ is a cutting element,
         we have $\beta \le \alpha$ or $\beta \ge \alpha$. If $\beta \le \alpha$,
         then $\tilde {f} (\vb{x}) = \tilde{f} (\vb{y})$.
          If $\beta \ge \alpha$, then since $f$ is congruence preserving,
         $f ( \vb{x} / \alpha ) \equiv_{\beta/\alpha} f ( \vb{y} / \alpha )$.
         Therefore,  $r ( f (\vb{x} / \alpha)) / \alpha \equiv_{\beta/\alpha}
                      r ( f (\vb{y} / \alpha)) / \alpha$, and thus
          $(r ( f (\vb{x} / \alpha)), r (f (\vb{y} / \alpha))) \in \beta$,
          which
          yields $(\tilde{f} (\vb{x}), \tilde{f} (\vb{y})) \in \beta$. \qed

    Since the function $\tilde{f}$ constructed in Proposition~\ref{pro:cut1}
    satisfies $\tilde{f}^{\A / \alpha} = f$, the last two propositions put
    together yield:
   \begin{lem}  \label{lem:factor}
     Let $\ab{A}$ be an algebra, and let $\alpha$ be a congruence of $\ab{A}$ that
     cuts the lattice $\Con (\ab{A})$.
     If $\Comp (\ab{A})$ is finitely generated, then $\Comp (\ab{A}/\alpha)$ is finitely generated.
   \end{lem}

\section{Preliminaries from lattice theory}\label{sec:lattices}
In this section, we recall the notion of \emph{perspectivity} 
in lattice theory \cite{Gr:GLTS} and its relation to the 
commutator operation of an algebra.
When $|\IntI [a,b]|=2$, we write $a \prec b$ and say that $\IntI
[a,b]$ is a \emph{prime interval} and that an element $a$ \emph{is
covered by} an element $b$ or that $b$ \emph{covers} $a$.
We say that an interval $\IntI [a,b]$ \emph{transposes up} to the
interval $\IntI [c,d]$ (or $\IntI [c,d]$ \emph{transposes down} to
$\IntI [a,b]$) if $a=b\wedge c$ and $d=b\vee c$, and we write
$$\IntI [a,b]\nearrow \IntI [c,d] \mbox{ or } \IntI [c,d]\searrow
\IntI [a,b].$$ This is shown on the picture below. The relation
$\sim:=\nearrow\cup\searrow$ is called the \emph{perspectivity
relation}. The transitive closure of $\sim$ is called
\emph{projectivity} and it is denoted by $\leftrightsquigarrow$.

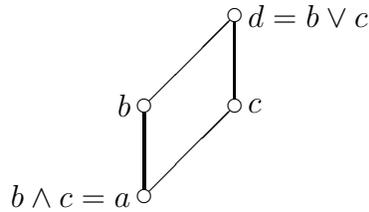
\begin{center}
\setlength{\unitlength}{0.06mm}
\begin{picture}(330,500)(-15,-65)
\put(100,200){\circle{30}} \put(100,0){\circle{30}}
\put(300,200){\circle{30}} \put(300,400){\circle{30}}
\put(111,211){\line(1,1){178}} \put(111,11){\line(1,1){178}}
\thicklines
\put(100,15){\line(0,1){170}} 
\put(300,215){\line(0,1){170}}
\put(70,200){\makebox(0,0)[r]{$b$}}
\put(330,200){\makebox(0,0)[l]{$c$}} 
\put(70, 0){\makebox(0,0)[r]{$b\wedge c = a$}}
\put(330,400){\makebox(0,0)[l]{$d=b\vee c$}}
\end{picture}
\end{center}

\begin{de}
   Let $\ob{L}$ be a lattice. Then
       $\ob{L}$ satisfies the condition $\AP$ (adjacent projectiveness)
     if for every $\alpha , \beta , \gamma \in \ob{L}$,
      with $\alpha \prec \beta$ and $\alpha \prec \gamma$ we have that $\IntI [\alpha, \beta] \leftrightsquigarrow
\IntI [\alpha, \gamma]$.
\end{de}
It is easy to see that every $\ob{M}_2$-free modular lattice has the property $\AP$.
In general, a finite modular lattice with $\AP$ need not be $\ob{M}_2$-free, but from
\cite[Proposition~4.1(1)]{Ai:TNOC} we obtain that the congruence lattice of a finite expanded group
has $\AP$ if and only if it is $\ob{M}_2$-free.  In Theorem~\ref{thm:marijana}, we will
see that finite modular lattices with $\AP$ can be nicely described geometrically.

In a finite lattice $\ob{L}$, an element $a\in \ob{L}$ is called \emph{meet irreducible} if $a<\bigwedge\{c\in \ob{L} \mid c>a\}$.
If $a$ is meet irreducible, we define $a^+:=\bigwedge\{c\in \ob{L} \mid c>a\}.$ Then $a\prec a^+.$ An element $b\in \ob{L}$ is called
\emph{join irreducible} if $b>\bigvee\{c\in \ob{L} \mid c<b\}.$ If $b$ is join irreducible, we define
$b^-:=\bigvee\{c\in \ob{L} \mid c<b\},$ and then $b^-\prec b.$ The following statement is a basic fact.

\begin{lem}\label{glupa}
Let $\ob{L}$ be a finite
lattice and let $a,b\in \ob{L}$ be such that $b\nleq a$. If $c\in \ob{L}$ is maximal with the property
$c\geq a$ and $c\ngeq b$, then $c$ is meet irreducible.
\end{lem}

\begin{proof}
Seeking a contradiction, we suppose that $c=\bigwedge\{d\in \ob{L}\mid d>c\}.$ By the maximality of $c$, every element $d$ with
$d>c$ satisfies $d\geq b$. Thus $c\geq b$, which is
a contradiction.
\end{proof}

The next two lemmas are restatements of \cite[p.35, Remarks 4.6]{FM:CTFC}.
\begin{lem}\label{lemica1}
Let $\ob{L}$ be a congruence lattice of a Mal'cev algebra
$\ab A$ and let $\alpha,\beta,\gamma,\delta\in \ob{L}$ be such that $\IntI [\alpha,\beta]\leftrightsquigarrow \IntI [\gamma,\delta]$. Then for all $\varepsilon\in \ob{L}$ we have that
$[\varepsilon,\beta]\leq \alpha \mbox{ if and only if }[\varepsilon,\delta]\leq \gamma$.
\end{lem}

\begin{proof}
First, let us suppose that these intervals are perspective, for example, $\IntI [\alpha,\beta]\nearrow \IntI [\gamma,\delta]$.
If $[\varepsilon,\beta]\leq \alpha$, then $[\varepsilon,\delta]= [\varepsilon,\beta\lor\gamma]=[\varepsilon,\beta]\lor[\varepsilon,\gamma]\leq \alpha\lor \gamma=\gamma.$
Next, if $[\varepsilon,\delta]\leq \gamma$, then $[\varepsilon,\beta]\leq [\varepsilon,\delta]\leq \gamma$ and $[\varepsilon,\beta]\leq \beta$. Thus, $[\varepsilon,\beta]\leq \gamma\land \beta=\alpha.$
The result now follows from the fact that projectivity is the transitive closure of perspectivity.
\end{proof}

\begin{lem}\label{lemica2}
Let $\ob{L}$ be a congruence lattice of a Mal'cev algebra
$\ab A$ and let $\alpha,\beta,\gamma,\delta\in \ob{L}$ be such that $\IntI [\alpha,\beta]\leftrightsquigarrow \IntI [\gamma,\delta]$. Then
$[\beta,\beta]\leq \alpha\mbox{ if and only if }[\delta,\delta]\leq \gamma$.
\end{lem}

\begin{proof}
As in the proof of the previous lemma, it is enough to consider
perspective intervals.
Let $\IntI [\alpha,\beta]\nearrow \IntI [\gamma,\delta]$. If $[\beta,\beta]\leq \alpha$, then $[\delta,\delta]=[\gamma\lor\beta,\gamma\lor\beta]\leq\gamma\lor [\beta,\beta]\leq\gamma\lor\alpha=\gamma.$ Next, if $[\delta,\delta]\leq \gamma$, then $[\beta,\beta]\leq[\delta,\delta]\leq \gamma$ and $[\beta,\beta]\leq\beta$. Thus, $[\beta,\beta]\leq \gamma\land\beta=\alpha$.
\end{proof}

Each congruence of a modular lattice of finite height is completely determined
by the projectivity classes of prime intervals that it collapses. In particular,
we have:
\begin{pr}(cf. \cite[Theorem 4.2]{F:ICRO})\label{prosta.mreza}
A modular lattice of finite height is simple if and only if every two prime intervals are projective.
\end{pr}

\begin{lem}\label{abelian}
Let $\A$ be a finite Mal'cev algebra and let $\ob{L}$ be the congruence lattice of $\A$.
If $\ob{L}$ is simple and $|\ob{L}|\geq 3$, then for every $\varphi,\psi\in \ob{L}$ such that
$\varphi\prec\psi$ we have $[\psi,\psi]\leq\varphi$.

\end{lem}

\begin{proof}
First we show that there exists a meet irreducible element $\eta\in \ob{L}$ such that $[\eta^+,\eta^+]\leq \eta$. Let $\alpha,\beta\in \ob{L}$ be such that $\alpha\prec\beta$. If $\eta$ is a maximal with the properties $\eta\geq \alpha, \eta\ngeq \beta$,
then $\eta$ is a meet irreducible element of $\ob{L}$, by Lemma \ref{glupa}. We will show that $[\eta^+,\eta^+]\leq \eta$. Since
$|\ob{L}|\geq 3$, there has to be another meet irreducible element $\theta$, and by simplicity
$\IntI [\eta,\eta^+]\leftrightsquigarrow \IntI [\theta,\theta^+]$. Therefore
by Proposition~\ref{prosta.mreza},
 there exist meet irreducible elements
$\eta=\eta_0,\eta_1,\dots,\eta_k=\theta$ and $\alpha_1,\dots,\alpha_k, \beta_1,\dots,\beta_k\in \ob{L}, k\in \mathbb{N}$ such that
$$\IntI [\eta,\eta^+]\searrow \IntI [\alpha_1,\beta_1]\nearrow \IntI [\eta_1,\eta_1^+]\searrow\cdots
\nearrow \cdots \searrow  \IntI [\alpha_k,\beta_k]\nearrow \IntI [\theta,\theta^+].$$
Let $s\in \{1,\dots,k\}$ be minimal such that $\eta_s\neq\eta$. Then $\eta_s$ is incomparable to $\eta$.
Let $\gamma$ be the largest element
$\gamma'$ with the property $[\gamma',\eta^+]\leq \eta$. By Lemma \ref{lemica1}, $\gamma$ is also the largest $\gamma''$ such that
$[\gamma'',\eta_s^+]\leq \eta_s.$ From $[\eta,\eta^+]\leq \eta$ and $[\eta_s,\eta_s^+]\leq \eta_s$, we conclude that $\gamma\geq\eta$
and $\gamma\geq \eta_s$. Therefore, $\gamma\geq \eta_s\lor\eta>\eta$, and so $\gamma\geq \eta^+$. Then
$[\eta^+,\eta^+]\leq[\gamma,\eta^+]\leq\eta.$
Now for proving the Lemma, we choose $\varphi$ and $\psi \in \ob{L}$
with $\varphi\prec\psi$.
Since $\ob{L}$ is simple,
Proposition \ref{prosta.mreza} implies $\IntI [\varphi,\psi]\leftrightsquigarrow \IntI [\eta,\eta^+]$. Then by
Lemma \ref{lemica2}, we obtain $[\psi,\psi]\leq\varphi$.
\end{proof}

\section{Lattices with (AP)} \label{sec:aplatt}

In Section~\ref{sec:mot}, we observed that the congruence lattice of
a $p$-group cannot have a subinterval isomorphic to the $4$-element
Boolean lattice, and we have called such lattices $\ob{M}_2$-free.
Every modular $\ob{M}_2$-free lattice satisfies $\AP$, and modular
$\AP$-lattices of finite height will be described now.
Actually, in this section, we prove that a modular lattice of finite height that satisfies
$\AP$ is a coalesced ordered sum
of simple lattices.
We will denote the height of an element $a$ in a modular lattice $\ob{L}$ with $ht(a)$.

\begin{lem}\label{lema1} \label{lem:simple}
Let $\mathbb{L}$ be a modular lattice of finite height with {\rm (AP)}, and let $\sigma \in \ob{L}$ be minimal among the nonzero cutting elements of $\ob{L}$.
 Then $\IntI [{\mathbf 0},\sigma]$ is simple.
\end{lem}

\begin{proof}
For an atom $\bar{\beta}$ in $\IntI [{\mathbf 0},\sigma]$, let us
call the interval $\IntI [{\mathbf 0},\bar{\beta}]$ a basic interval
in $\IntI [{\mathbf 0},\sigma]$. Note that all basic intervals are
projective (using (AP)).

Assume that $\IntI [{\mathbf 0},\sigma]$ is not simple. Then by
Proposition \ref{prosta.mreza} there exists an element $\beta$ of
minimal height in $\IntI [{\mathbf 0},\sigma]$ with the property
that there exists an $\alpha\in \ob{L}$ such that
$\alpha \prec\beta$ and  $\IntI [\alpha,
\beta]$ is not projective to $\IntI [{\mathbf 0},\bar{\beta}]$.
We claim that $\beta$ is join irreducible. Let us suppose the opposite that there exists an $\alpha_1 \ne \alpha$ covered by
$\beta$. Since $\mathbb{L}$ is modular,
 $\alpha \land \alpha_1 \prec \alpha_1$ and $\alpha \land \alpha_1
\prec \alpha$, which implies  that $\IntI [\alpha, \beta] \searrow \IntI [\alpha \land \alpha_1, \alpha_1]$. Now $\alpha_1$ has smaller
height than $\beta$,
which is a contradiction
to the minimality of $\beta$.

Since $\sigma$ is a minimal cutting element and $\beta^- \prec \beta \leq \sigma$ , $\beta^-$ does not cut $\mathbb{L}$. Our
next step is to demonstrate that in this case there exists a $\gamma \ne \beta^-$ such that $ht(\gamma)=ht(\beta^-)$. Since
$\beta^-$ does not cut $\mathbb{L}$, there is an element $\delta$, incomparable with $\beta^-$. Since $\ob{L}$ is
a modular lattice, \cite[Theorem~2.37]{MMT:ALVV} implies that every maximal chain from
$\beta^-\land\delta$ to $\beta^-\lor\delta$ has the same length. Let $C$ be a maximal chain
in $\IntI [\beta^- \meet \delta, \beta^- \join \delta]$ that contains $\delta$.
In $C$, we find an element $\gamma$ with $ht (\gamma) = ht (\beta^-)$. Since $\gamma$ is comparable with $\delta$, we have
$\gamma \neq \beta^-$. We choose $\tilde{\gamma}$ such that
$\gamma \meet \beta^- \prec \tilde{\gamma}  \leq \gamma$.
Since  $\IntI [\gamma \land \beta^-, \tilde{\gamma}] \nearrow \IntI [\beta^-, \tilde{\gamma} \lor \beta^-]$, and  since projective
intervals in a modular lattice are isomorphic (\cite[Corollary~2.28]{MMT:ALVV}), we have $\beta^- \prec \tilde{\gamma} \lor \beta^-$.
Note that since $ht(\tilde{\gamma}) \leq ht(\gamma) < ht(\beta)$,  we have $\tilde{\gamma} \ne \beta$. Using this fact and the join
irreducibility of $\beta$, we obtain that $\tilde{\gamma} \lor \beta^- \ne \beta$. Therefore,
$$\IntI [\gamma \land \beta^-, \tilde{\gamma}] \nearrow \IntI [\beta^-, \tilde{\gamma} \lor \beta^-] \stackrel{(AP)}{\leftrightsquigarrow}
\IntI [\beta^-,\beta],$$ which means that $\IntI [\gamma \land
\beta^-, \tilde{\gamma}]$ is not projective with a basic interval.
As $ht(\tilde{\gamma}) < ht(\beta)$, this yields a contradiction to
the choice of $\beta$. Thus, $\IntI [{\mathbf 0},\sigma]$ is simple.
\end{proof}

If we have a lattice $\mathbb{L}$ and elements $\alpha_0,\dots,
\alpha_{n}$ that cut $\mathbb{L}$ such that ${\mathbf 0}=\alpha_0<
\alpha_1 < \dots < \alpha_n= {\mathbf 1}$ and every sublattice
$\IntI [\alpha_i,\alpha_{i+1}], i\in \{0,\dots ,n-1\}$ is simple,
then we say that $\mathbb{L}$ is a \emph{coalesced ordered sum of
simple lattices}.

\begin{thr}\label{thm:marijana}
Let $\mathbb{L}$ be a modular lattice of finite height. The following are equivalent:
\begin{enumerate}
 \item $\mathbb{L}$ has {\rm (AP)},
 \item $\mathbb{L}$ is a coalesced ordered sum of simple lattices.
\end{enumerate}
\end{thr}

\begin{proof}
Every coalesced ordered sum of simple modular lattices has (AP) by Proposition \ref{prosta.mreza}. For the opposite direction we use
Lemma \ref{lema1} and induction on the number of elements that cut the lattice.
\end{proof}

\section{Finite generation in Mal'cev algebras}\label{Malcev}

In this section, we will characterize which finite Mal'cev algebras with
a simple congruence lattice have the property that their clone of
congruence preserving functions is finitely generated.
Hence this section will close with the proof of
Theorem \ref{thm:malcev}; before this, we need
some preparation.

\begin{lem}\label{nilpotent}
Let $\A$ be a finite Mal'cev algebra and let $\ob{L}$ be the congruence lattice of $\A$. If $\ob{L}$ is simple and $|\ob{L}|\geq 3$, then $\A$ is nilpotent.
\end{lem}

\begin{proof}
Let $\varphi\in \ob{L}$ be an element such that $\varphi\prec
{\mathbf 1}$. By Lemma \ref{abelian}, $[{\mathbf 1},{\mathbf 1}]\leq
\varphi$. Let $\alpha,\beta\in \ob{L}$ be such that $\beta\neq
{\mathbf 0}$ and $\alpha\prec\beta$. Since $\IntI [\varphi,{\mathbf
1}]\leftrightsquigarrow \IntI [\alpha,\beta]$, using Lemma
\ref{lemica1} with $\delta=\varepsilon={\mathbf 1}$ and
$\gamma=\varphi$, from $[{\mathbf 1},{\mathbf 1}]\leq \varphi$ we
conclude that $[{\mathbf 1},\beta]\leq\alpha$. Therefore, for every
nonzero element $\beta\in \ob{L}$ we have that
\begin{equation}\label{j-na***}[{\mathbf 1},\beta]< \beta.\end{equation}

Now we define a sequence $(\gamma_i)_{i\in \mathbb{N}}$ of elements
of $\ob{L}$ in the following way: $\gamma_1=[{\mathbf 1},{\mathbf
1}],$ and for $i>1$, $\gamma_i=[{\mathbf 1},\gamma_{i-1}]$. Then for
all $i > 1$, we have $\gamma_i \le \gamma_{i-1}$. Since $\ob{L}$ is
finite,
 there exists a $k > 1$ such that $\gamma_k = \gamma_{k-1}$, and thus $[{\mathbf 1},\gamma_{k-1}]=\gamma_{k-1}$. From (\ref{j-na***}) we conclude that $\gamma_{k-1}={\mathbf 0}$.
 Thus $\A$ is nilpotent.
\end{proof}

\begin{lem}\label{ppo}
Let $\A$ be a finite Mal'cev algebra and let $\ob{L}$ be the congruence lattice of $\A$. If $\ob{L}$ is simple and $|\ob{L}|\geq 3$,
then there exist a prime number $p$ and an $n\in \mathbb{N}$ such that $|A|=p^n$.
\end{lem}
 \begin{proof}
  Many of the arguments in this proof come from \cite{FM:CTFC}.
 We choose an element $a \in A$, and a minimal congruence $\beta \succ {\mathbf 0}$ of $\ab{A}$.
 We first establish that the cardinality of $a / \beta$ is a prime power:
 let $d$ be the Mal'cev term of $A$.
 Since $[\beta, \beta] = {\mathbf 0}$, the set $(a / \beta, +_a)$ with $x+_a y = d(x,a,y)$
 is a module over the ring $(\{ p|_{a/ \beta} \setsuchthat p \in \Pol_1 \ab{A}, p(a) = a \},+_a, \circ)$
 and by the minimality of $\beta$, this module is simple. Hence ring theory yields
 that $a / \beta$ is of prime power cardinality.

 Let $\gamma \prec \delta$ be a covering pair in $\ob{L}$.
 We first show
 \begin{equation} \label{eq:allsame}
     |\{ b \setsuchthat (b,a) \in \beta \} | =
     |\{ b/\gamma \setsuchthat (b,a) \in \delta|.
 \end{equation}
 To this end, let $\alpha_1 \prec \beta_1$ and $\gamma_1 \prec \delta_1$ be
 elements of $\ob{L}$ such that $\IntI [\alpha_1, \beta_1] \nearrow \IntI [\gamma_1, \delta_1]$.
 We consider $\psi : \{ b/ \alpha_1 \setsuchthat (b, a) \in \beta_1 \} \to
                     \{ b/ \gamma_1 \setsuchthat (b, a) \in \delta_1 \}$, $\psi (b/\alpha_1) :=
                 b / \gamma_1$ for all $b \in a/\beta_1$.
 Then it is easy to see that $\psi$ is well-defined and injective. Since
 $\beta_1 \circ \gamma_1 = \delta_1$, $\psi$ is also surjective.
 Now~\eqref{eq:allsame} follows from the fact that all prime intervals of $\ob{L}$
 are projective.

 Next we show that for all $a_1, a_2 \in A$,
 \begin{equation} \label{eq:uniform}
    |\{ b/ \gamma \setsuchthat (b, a_1) \in \delta \}| =
    |\{ b/ \gamma \setsuchthat (b, a_2) \in \delta \}|.
 \end{equation}
 To this end, we observe that $\delta/\gamma$ is a minimal congruence of
 $\ab{A} / \gamma$. Now for a minimal congruence $\alpha_2$ of a nilpotent
 Mal'cev algebra $\ab{B}$ with Mal'cev term $d$ and $b_1, b_2 \in B$, the mapping
 $\psi : b_1 / \alpha_2 \to b_2 / \alpha_2$, $x \mapsto d(x, b_1, b_2)$ is
 injective: suppose $d(x_1, b_1, b_2) = d(x_2, b_1, b_2)$ with $x_1, x_2 \in b_1 / \alpha_2$, and let
                             $t(x,y) := d( d(x, y, b_2), d(b_2, y, b_1), b_1)$.
 Then $t (b_1, b_1) = b_1  = t(b_1, b_2)$. Since, by nilpotence, $[\alpha_2, {\mathbf 1}]_{\ab{B}} = {\mathbf 0}$,
 the term condition yields $t (x_1, b_1) = t (x_1, b_2)$ and
 $t(x_2, b_1) = t(x_2, b_2)$.
  Since $t(x_1, b_1) = d ( d ( x_1, b_1, b_2), b_2, b_1)$,
        $t(x_1, b_2) = x_1$,
        $t(x_2, b_1) = d ( d ( x_2, b_1, b_2), b_2, b_1)$,
        $t(x_2, b_2) = x_2$, we obtain $x_1 = x_2$. (This argument also
    appears in
 \cite[Proposition~2.7(4)]{Ai:TPFO2}.)
 Hence all congruence classes of $\alpha_2$ of $\ab{B}$ have the same size.
 This completes the proof of~\eqref{eq:uniform}.

 Now let ${\mathbf 0} = \gamma_0 \prec \gamma_1 \prec \ldots \prec \gamma_m = {\mathbf 1}$  be a maximal chain
 in $\ob{L}$. By induction on $i$, we prove that for all $b \in A$,
  we have $|b / \gamma_i| = q^i$, where
 $q := |a / \beta|$. For $i = 0$, the statement is obvious. Now let
 $i \ge 1$. Then $b / \gamma_i = \bigcup \{ c / \gamma_{i-1} \setsuchthat
                                              (c, b) \in \gamma_{i} \}$.
 By the induction hypothesis, each class $c / \gamma_{i-1}$ has exactly
 $q^{i-1}$ elements. Hence $|b/ \gamma_i| = k q^{i-1}$, where
 $k = | \{c / \gamma_{i-1} \setsuchthat c \in b/ \gamma_i \}|$. By~\eqref{eq:uniform},
 $k = | \{c / \gamma_{i-1} \setsuchthat c \in a / \gamma_i \}|$. Now by~\eqref{eq:allsame},
 $k = | a / \beta| = q$. Thus $|b / \gamma_i| = q^i$, which completes the induction step.

  Altogether, $|A| = |a / \gamma_m| = q^m$.
 \end{proof}

\begin{lem}\label{Malcev21}
Let $\mathbf{A}$ be a finite algebra with a Mal'cev term, and let
$\mathbb{L}$ be the congruence lattice of $\mathbf{A}$. If
$\mathbb{L}$ does not split or $|\ob{L}| \le 2$, then
$\mathrm{Comp}(\mathbf{A})$ is finitely generated.
\end{lem}

{\it Proof:} Suppose that $\mathbb{L}$ does not split. Let
$\ABar=(A, \mathrm{Comp}(\mathbf{A}))$ and let $\mathbb{L}'$
be the congruence lattice of $\ABar$. Note that
\begin{equation}\label{eq1} \mathbb{L}=\mathbb{L}'\mbox{ and }
\mathrm{Pol}(\ABar)=\mathrm{Comp}(\mathbf{A}).
\end{equation}
From Proposition \ref{Lema3.3.AM:SOCO} we know that $\ABar$ is
supernilpotent because $\mathbb{L}'$ does not split. Then by
Proposition \ref{pr6.18.AM:SAOH}, for some $k\geq 1$ and a Mal'cev
term $m$ in $\ABar$, we have
\begin{equation}\label{eq2}
\Clo^{A} (\mathrm{Pol}_k(\ABar) \cup
\{m\} ) = \mathrm{Pol}(\ABar).
\end{equation}
From \eqref{eq1} and \eqref{eq2} we conclude that
$\mathrm{Comp}(\A)$ is finitely generated.\qed

Now we have all the tools that are needed for proving Theorem~\ref{thm:malcev}.

\emph{Proof of Theorem~\ref{thm:malcev}}:
$(1)\Rightarrow (2)$: Let $\mathrm{Comp}(\mathbf{A})$ be finitely
generated and let $G$ be a finite set of generators for
$\mathrm{Comp}(\A)$. Then the algebra $\mathbf{A}'=(A, G)$ has finite
type and the congruence lattice of $\mathbf{A}'$ is $\mathbb{L}$.
If $|\ob{L}| = 2$, we are done, hence we assume $|\ob{L}| \ge 3$.
Since $\ob{L}$ is a simple lattice, Lemma
\ref{nilpotent} implies that the algebra $\A'$ is nilpotent, and by Lemma
\ref{ppo}, it is of prime power order. From Proposition
\ref{Lema7.6.AM:SAOH} we can see that in this case $\A'$ is
supernilpotent.
Seeking a contradiction, we suppose that $\ob{L}$ has a
splitting pair $(\delta,\epsilon)$. Then
$$\ob{L}=\IntI [{\mathbf 0},\delta]\cup \IntI [\epsilon, {\mathbf 1}].$$
Since $\epsilon>{\mathbf 0}$, there exist $o,c \in A, c \neq o,$
such that $o\equiv_\epsilon c$. Now for every $n \in \N$, we will
construct an absorbing polynomial $c_n$ of $\ab{A}'$ of essential arity $n$. For
$n\in \mathbb{N}$, let $c_n:A^n\rightarrow A$ be defined by
$$c_n(x_1,\dots, x_n)=\left\{\begin{array}{ll}
c& \text{ if } x_1,\dots, x_n\in o/\delta\\
o& \mbox{otherwise}
\end{array}\right.$$
for $x_1,\dots, x_n\in A$. By Proposition~\ref{pro:de}, $c_n$ is
congruence preserving. Therefore, $c_n\in \mathrm{Comp}(\A)$. Since
$\Comp (\ab{A}) = \Clo^A (G)$, $c_n$ is a term function on $\A'$,
and thus $c_n \in \Pol (\A')$. Now we will prove that the essential
arity of $c_n$ is $n$, which means that for each $i\in \{1,\dots
,n\}$, $c_n$ depends on its $i$-th argument. Since $\delta<{\mathbf
1}$, there exists $d\not\in o/\delta$, and then
$c_n(o,\dots,o) = c\neq o=
c_n(o,\dots,o, d ,o,\dots,o)$, with $d$ at the $i$th place.
By its definition, the function $c_n$
is absorbing at $(d,d,\dots, d)$ with value $o$, and
by Proposition~\ref{Cor6.17.AM:SAOH}, the existence of these functions $c_n$
contradicts the supernilpotence of $\A'$.

$(2)\Rightarrow (1)$: Lemma \ref{Malcev21}. \qed

\section{Abelian groups} \label{sec:abelian}
The theory developed so far allows us to characterize those
finite abelian groups that have a finitely generated clone
of congruence preserving functions. We first deal with $p$-groups.
We notice that the congruence lattice of an abelian group is isomorphic
to its subgroup lattice.

\begin{lem} \label{lem:splitabelian}
    Let $p$ be a prime, let $n \in \N$, and let $\alpha_1, \ldots, \alpha_n \in \N$ be
     such that $\alpha_1 \ge \alpha_2 \ldots \ge \alpha_n$.
    Then the subgroup lattice of $G := \Z_{p^{\alpha_1}} \times \cdots \times \Z_{p^{\alpha_n}}$ splits
    if and only if $n = 1$ or ($n \ge 2$ and $\alpha_1 > \alpha_2$).
\end{lem}
 \emph{Proof:}
      For the ``if''-direction, we first consider the case $n = 1$. Then  $G = \Z_{p^{\alpha_1}}$,
   and we choose $D := \langle p \rangle$, and $E := \langle p^{\alpha_1 - 1} \rangle$.
   Then $(D, E)$ is a splitting pair of the subgroup lattice of $G$.
   Now we assume $n \ge 2$ and $\alpha_1 > \alpha_2$. Then we set
   \begin{equation*}
      \begin{array}{rcl}
           D & := & \{ x \in G \setsuchthat p^{\alpha_1 - 1} x = 0 \} \\
           E & := & \{ p^{\alpha_1 - 1} x \setsuchthat x \in G \}.
       \end{array}
   \end{equation*}
   We claim that $(D,E)$ is a splitting pair of the subgroup lattice of $G$.
   To this end, let $N$ be a subgroup of $G$ with $N \not\le D$. Then $N$ contains an element
   $(x_1, \ldots, x_n)$ with $p^{\alpha_1 - 1} x_1 \neq 0$. Hence
   $p^{\alpha_1 - 1} (x_1, \ldots, x_n) = (p^{\alpha_1 - 1} x_1, 0, 0, \ldots, 0) \in N$.
   Therefore, $N$ contains the generator $(p^{\alpha_1 - 1}, 0, 0, \ldots, 0)$ of the cyclic
   group $E$, and we have $E \le N$.

   For the ``only if''-direction, we assume that $n \ge 2$ and $\alpha_1 = \alpha_2$, and we
   prove that the subgroup lattice of $G$ does not split.
   To this end, let $k \in \{1, \ldots, n\}$ be maximal with $\alpha_1 = \alpha_k$,
   let $H := (\Z_{p^{\alpha_1}})^k$, and let $K := \Z_{p^{\alpha_{k+1}}} \times \cdots \times \Z_{p^{\alpha_n}}$.
   Then $G = H \times K$.
   Now suppose that $(D, E)$ is a splitting pair of the subgroup lattice of $G$.
   We first prove
   \begin{equation} \label{eq:dk}
         \{0\} \times K \le D.
   \end{equation}
   Suppose  $\{0\} \times K \not\le D$. Then $\{0\} \times K  \ge E$. Since $E$ is a
   nontrivial subgroup of $G$, there is $k \in K \setminus \{0\}$ such that
   $(0, k) \in E$.
   We choose $x \in H$ with $\ord (x) = p^{\alpha_1}$. Then for every $r \in K$, the cyclic
   subgroup $\langle (x, r) \rangle$ of $H \times K$ does not contain $(0, k)$, and hence
   $\langle (x, r) \rangle \not\ge E$. Thus, by the splitting property,
   we have $\langle (x, r) \rangle \in D$.
   Hence for every $r \in K$, $(x, r) - (x, 0) \in D$, and thus $\{0\} \times K \le D$.
   This completes the proof of~\eqref{eq:dk}.
   Next, we prove that there is $x_1 \in H$ such that $\ord (x_1) = p^{\alpha_1}$
   and $(x_1,0) \not\in D$. Suppose that for all elements $x_2$ of order
   $p^{\alpha_1}$, we have $(x_2, 0) \in D$. Since these elements generate $H$,
   we obtain $H \times \{0\} \le D$. Together with~\eqref{eq:dk}, we obtain
   $H \times K = D$, which is impossible because $D$ comes from a splitting pair.
   Hence such an $x_1$ exists.
   By \cite[4.2.7]{Ro:ACIT}, there is a subgroup $H_1$ of $H$ such that
   $H$ is an inner direct sum $\langle x_1 \rangle + H_1$.
   Since $H = (\Z_p^{\alpha_1})^k$ and $k \ge 2$, $H_1$ contains an element
   $y_1$ of order $p^{\alpha_1}$.
   Let us first assume $(y_1, 0) \not\in D$.
   Then $\langle (x_1, 0) \rangle \not\le D$, hence
    $E \le \langle (x_1, 0) \rangle$, and similarly,
    $E \le \langle (y_1, 0) \rangle$.
   Thus $E \le \langle (x_1, 0) \rangle \cap \langle (y_1, 0) \rangle \le
                      (\langle x_1 \rangle \times \{0\}) \cap (H_1 \times \{0\}) = \{0\} \times \{0\}$ because
    the sum $H = \langle x_1 \rangle + H_1$ is direct. This is not possible for a splitting pair $(D, E)$.
    The second case is $(y_1, 0) \in D$.
    Then, since $(x_1, 0) \not\in D$, we have  $(x_1 + y_1, 0) \not\in D$.
     Reasoning as above, we obtain
     $E \le (\langle x_1 + y_1 \rangle \cap \langle x_1 \rangle) \times \{0\}$.
     Now let $z = a (x_1 + y_1) = b x_1$ be an element in
      $\langle x_1 + y_1 \rangle \cap \langle x_1 \rangle$, where $a, b \in \Z$.
    Then
    $a y_1 = (b - a) x_1$, hence $a  y_1 = 0$. Since $y_1$ is of order $p^{\alpha_1}$,
    we obtain $p^{\alpha_1} \mid a$, and therefore $a (x_1 + y_1) = 0$, which implies $z = 0$.
    Altogether $E = \{0\}$, which is impossible if $(D, E)$ is a splitting pair. \qed

\begin{thm} \label{thm:abgroups}
     Let $p$ be a prime, let $n \in \N$, and let $\alpha_1, \ldots, \alpha_n \in \N$ be
     such that $\alpha_1 \ge \alpha_2 \ldots \ge \alpha_n$.
    Then the clone of congruence preserving functions of the group
    $G := \Z_{p^{\alpha_1}} \times \cdots \times \Z_{p^{\alpha_n}}$
     is finitely generated if and only if
     $n = 1$ or ($n \ge 2$ and $\alpha_1 = \alpha_2$).
\end{thm}

\emph{Proof:}
    For proving the ``if''-direction, let us first assume that
    $G$ is cyclic.  Then  subgroup lattice of $G$ is a chain.
    Hence the subgroup lattice of $G$ is distributive, and so
    the clone of congruence preserving functions of $G$ is generated
    by all binary congruence preserving functions (cf. \cite[Proposition~5.2 (1)]{Ai:OHAH};
     the result also follows from \cite[Corollary~11.7]{AM:TOPC}).
    If  $G$ is not cyclic, then  $n \ge 2$ and
    $\alpha_1 = \alpha_2$.
     Now from N\"obauer's characterization of affine complete groups
    \cite[Satz~5]{No:UDAV} we obtain that $G$ is affine complete, which implies
    that the clone of congruence preserving functions is finitely generated.
     Without resorting to \cite{No:UDAV}, we proceed as follows:
 By Lemma~\ref{lem:splitabelian}, the subgroup lattice
    of $G$ does not split.
    Furthermore, the subgroup lattice of every finite $p$-group satisfies $\AP$.
    Since the congruence lattice of $G$ does not split, $\{0\}$ and $G$ are the
    only cutting elements of the subgroup lattice. Thus Lemma~\ref{lema1}
    implies that the subgroup lattice of $G$ is simple.
     Now Theorem~\ref{thm:malcev} yields
    that the clone of congruence preserving functions of $G$ is finitely generated.

      For proving the ``only if''-direction, we assume that $G$ has a finitely
     generated clone of congruence preserving functions. Suppose
    that the subgroup lattice of $G$ has a cutting element $T$ other than $\{0\}$ and $G$.
    Then either $G$ is cyclic, and the proof is complete, or, as a finite noncyclic
     abelian group, $G$ is isomorphic
    to a direct sum of two proper subgroups $A$ and $B$. Then if
    $A \le T$, $B \le T$, we obtain $G \le T$; if $A \le T$, $B \ge T$, we obtain
     $G = B$; if $A \ge T$, $B \le T$, we obtain $G = A$; and if
     $A \ge T$, $B \ge T$, we obtain $\{0\} = A \cap B \ge T$; none of these can occur.
     Thus $\{0\}$ and  $G$ are the only cutting elements. Since $G$ is a $p$-group, its
    normal subgroup lattice satisfies $\AP$, and therefore its subgroup lattice
    is simple by Lemma~\ref{lema1}. Now Theorem~\ref{thm:malcev} implies
    that either the subgroup lattice is $\{\{0\}, G\}$, or
    the subgroup lattice does not split.
      In the first case, $G$ is cyclic of prime order, and hence $n=1$.
      If the subgroup lattice of $G$ does not split then from Lemma~\ref{lem:splitabelian},
    we obtain that
      $n \ge 2$ and $\alpha_1 = \alpha_2$.  \qed

  Using N\"obauer's description of finite affine complete abelian groups,
  we can now derive
  Theorem~\ref{thm:abelian}.

   \emph{Proof of Theorem~\ref{thm:abelian}:}
      For the ``only if'' direction, we assume that the clone of congruence preserving
     functions of the finite abelian group $G$ is finitely generated.
      By Proposition~\ref{pro:dp2}, the clone of congruence preserving functions
     of each Sylow subgroup of $G$ is finitely generated. Let $S$ be a Sylow subgroup
     of $G$. Then by
      Theorem~\ref{thm:abgroups}, $S$ is cyclic or isomorphic to a group
      $(\Z_{p^{\alpha_1}})^k \times H$ with $\exp (H) \mid p^{\alpha_1 - 1}$ and $k \ge 2$.
       In the latter case, $S$ is affine complete by \cite[Satz~5]{No:UDAV}.
      For the ``if''-direction, we use Proposition~\ref{pro:dp2} and
       observe that the congruence preserving
      functions of a cyclic group are generated by the binary functions
      by \cite[Corollary~11.7]{AM:TOPC} or \cite[Proposition~5.2 (1)]{Ai:OHAH};
      and for an affine complete group, the
      clone of congruence preserving functions is generated by
     the (unary) constant operations and the fundamental operations
     of the group. \qed

\section{Restrictions of compatible functions}

Our next goal is to generalize the results from abelian groups to
arbitrary expanded groups. We will first derive necessary conditions
for $\Comp (\ab{A})$ to be finitely generated.
In expanded groups, we often work with ideals rather than congruences.
If $\ab{A}$ is an expanded group and $I$ is an ideal of $\ab{A}$,
we write $\bar{x} \equiv_I \bar{y}$ if and only if $x_i - y_i
\in I$ for all $i\in \{1,\dots,k\}$. Using this notation, we see
that a function $g \in A^{A^k}$ with $k\in\mathbb{N}$ is compatible
if for every ideal $I$ of $\A$ and for all $\bar{a}, \bar{b}\in A^k$
with $\bar{a}\equiv_I \bar{b}$, we have $f(\bar{a})\equiv_I
f(\bar{b})$.
Again, $\ABar$ denotes the algebra
$\algop{A}{\C{(\ab{A})}}$ that has all compatible functions on
$\ab{A}$ as its basic operations.
For a function $f \in \Comp_n (\ab{A})$ with $f(I^n) \subseteq I$, we write
$f|_I$ or $f|_{I^n}$ for the restriction of $f$ to $I^n$.
 Following \cite[p.6]{HM:TSOF}, we
write $\AI$ for the algebra that $\ABar$ induces on $I$; this means
 \[
     \AI = \algop{I}{\{ c|_{I^{ar (c)}} \setsuchthat c \in \Comp (\ab{A}), c (I^{ar(c)}) \subseteq I \}}.
 \]
  Now for deriving
  a necessary condition for $\Comp (\ab{A})$ being finitely generated,
  we will use induction on the height of the ideal lattice of $\ab{A}$. Let $I$ be a minimal cutting
   element of the ideal lattice of $\ab{A}$ with $I > {\mathbf 0}$. Then the factor algebra $\ab{A} / I$ and the
  induced algebra $\AI$
   have congruence lattices isomorphic to the intervals $\IntI [I, A]$ and $\IntI [{\mathbf 0}, I]$, respectively.
  In order to use some kind of inductive argument, we have to guarantee that $\ab{A}/I$ and $\AI$ also
  have a finitely generated clone of  congruence preserving functions.
  For the factor algebra $\A/I$, this is settled in Lemma~\ref{lem:factor},
  hence we will now
  treat $\ABar|_{I}$.
 Suppose that
   $f, f_1, \ldots, f_m \in \Comp (\ab{A})$ are congruence preserving functions
  of $\ab{A}$ that map $I^{n_i}$ into $I$
  and $f  \in \Clo^{A} (f_1, \ldots, f_m)$. Unfortunately, this does not imply
  that the restriction $f|_I$ lies in $\Clo^I (f_1|_I,\ldots, f_m|_I)$.
  However, if we allow to use all ``shifts'' of the $f_i$'s, we will be able to produce $f$ from
  compatible functions all of which map $I^k$ into $I$.
    We will now define these shifts.
    Let $\A$ be a
   finite
   expanded group, let $I$ be an ideal of $\A$, let $k:=|A/I|$, let $(0=s_1,\dots,s_k)$ be a transversal modulo $I$,
and let $r$ be a function that maps every element $a\in A$ to the representative of its $I$-class. Let $n\in \mathbb{N}$. For every $f\in A^{A^{n}}$ and for all $\alpha_1,\dots,\alpha_n \in \{s_1,\dots,s_k\}$ we define $T^n_{(\alpha_1,\dots,\alpha_n)}(f):A^n\rightarrow A$ in the following way:
$$T^n_{(\alpha_1,\dots,\alpha_n)}(f)(x_1,\dots,x_n):=f(x_1+\alpha_1,\dots,x_n+\alpha_n)-r(f(\alpha_1,\dots,\alpha_n))$$
for all $x_1,\dots,x_n\in A$.
Since $I$ is an ideal of $\A$ we have
\begin{equation}\label{IuI}
T^n_{(\alpha_1,\dots,\alpha_n)}(f)(I^n)\subseteq I
\end{equation}
for all $\alpha_1,\dots,\alpha_n\in A$.
Also if $f(I^n)\subseteq I$, since $0=r(f(0,\dots,0))$, we have
\begin{equation}\label{neMenja}
T^n_{(0,\dots,0)}(f)=f.
\end{equation}

\begin{lem}\label{12345}
Let $\A$ be a finite expanded group, let $I$ be an ideal of $\A$, let $k:=|A/I|$, let
$(0=s_1,\dots,s_k)$ be a transversal modulo $I$, and let $n,m\in\mathbb{N}$. For every $f\in A^{A^{n}}, g\in A^{A^{m}}$,
and for every $\vb{\alpha} \in \{s_1,\dots,s_k\}^{n+m-1}$ we have:
\begin{enumerate}
\item \label{it:ts1} $T^n_{(\alpha_1,\alpha_2,\dots,\alpha_n)}(\zeta f)=\zeta (T^n_{(\alpha_2,\dots,\alpha_n,\alpha_1)}(f));$
\item \label{it:ts2} $T^n_{(\alpha_1,\alpha_2,\alpha_3,\dots,\alpha_n)}(\tau f)=\tau (T^n_{(\alpha_2,\alpha_1,\alpha_3,\dots,\alpha_n)}(f));$
\item \label{it:ts3} if $n \ge 2$,
                     then $T^{n-1}_{(\alpha_1,\alpha_2,\dots,\alpha_{n-1})}(\Delta f)=\Delta (T^n_{(\alpha_1,\alpha_1,\alpha_2,\dots,\alpha_{n-1})}(f))$, and
                     for $n = 1$, $\Delta (T^1_{\alpha_1} (f)) = T^1_{\alpha_1} (\Delta f)$.
\item \label{it:ts5} $T^{n+1}_{(\alpha_1,\alpha_2,\dots,\alpha_{n+1})}(\nabla f)=\nabla (T^n_{(\alpha_2,\dots,\alpha_{n+1})}(f));$
\item \label{it:ts4} $T^{m+n-1}_{(\alpha_1,\dots,\alpha_{m+n-1})}(g\circ f)=T^m_{(\alpha_1,\dots,\alpha_m)}(g)\circ T^n_{(r(g(\alpha_1,\dots,\alpha_m)),\alpha_{m+1},\dots,\alpha_{m+n-1})}(f)$.
\end{enumerate}
\end{lem}

\begin{proof}
First of all, we see that
items~\eqref{it:ts1},~\eqref{it:ts2},~\eqref{it:ts3}
and~\eqref{it:ts5} hold if $n = 1$. Let $x_1, x_2,
\dots,x_{m+n-1}\in A$ and $\alpha_1,\alpha_2,\dots,\alpha_{m+n-1}\in
\{s_1,\dots,s_k\}$. Now for proving~\eqref{it:ts1}, we let $n \ge 2$
and compute
 \begin{eqnarray*}
& & T^n_{(\alpha_1,\alpha_2,\dots,\alpha_n)}(\zeta f)(x_1,x_2,\dots,x_n)=\\
&=& \zeta f(x_1+\alpha_1,x_2+\alpha_2,\dots,x_n+\alpha_n)-r(\zeta f(\alpha_1,\alpha_2,\dots,\alpha_n))=\\
&=& f(x_2+\alpha_2,\dots,x_n+\alpha_n,x_1+\alpha_1)-r(f(\alpha_2,\dots,\alpha_n,\alpha_1))=\\
&=& T^n_{(\alpha_2,\dots,\alpha_n,\alpha_1)}(f) (x_2,\dots,x_n,x_1)=\\
&=& \zeta (T^n_{(\alpha_2,\dots,\alpha_n,\alpha_1)}(f))(x_1,x_2,\dots,x_n).
\end{eqnarray*}
Analogously we obtain \eqref{it:ts2}, \eqref{it:ts3} and
\eqref{it:ts5} substituting $\zeta$ by $\tau,\Delta$ and $\nabla$
(and adjusting the variables).
Item~\eqref{it:ts4}:
 Let $\bar{\alpha}=(\alpha_1,\dots,\alpha_{m+n-1}), \bar{\alpha}_1=(\alpha_1,\dots,\alpha_m), \bar{\alpha}_2=(\alpha_{m+1},\dots,\alpha_{m+n-1}),\\ \bar{x}=(x_1,\dots,x_{m+n-1}), \bar{x}_1=(x_1,\dots,x_{m}), \bar{x}_2=(x_{m+1},\dots,x_{m+n-1})$. Since $f$ is a compatible function, we have
\(
   r(f(g(\bar{\alpha}_1),\bar{\alpha}_{2}))=r(f(r(g(\bar{\alpha}_1)),\bar{\alpha}_{2})).
\)
Now we obtain
\begin{eqnarray*}
& & T^{m+n-1}_{\bar{\alpha}}(g\circ f)(\bar{x})=\\
&=& (g\circ f)(\bar{x}+\bar{\alpha})-r((g\circ f)(\bar{\alpha}))=\\
&=& f(g(\bar{x}_1+\bar{\alpha}_1),\bar{x}_2+\bar{\alpha}_2)-r(f(g(\bar{\alpha}_1),\bar{\alpha}_2))=\\
&=& f(g(\bar{x}_1+\bar{\alpha}_1)-r(g(\bar{\alpha}_1))+r(g(\bar{\alpha}_1)),\bar{x}_2+\bar{\alpha}_2)
-r(f(g(\bar{\alpha}_1),\bar{\alpha}_2))=\\
&=& f\big(T^m_{\bar{\alpha}_1}(g)(\bar{x}_1)+r(g(\bar{\alpha}_1)),\bar{x}_2+\bar{\alpha}_2 \big)-r(f(r(g(\bar{\alpha}_1)),\bar{\alpha}_2))=\\
&=& T^n_{(r(g(\bar{\alpha}_1)),\bar{\alpha}_2)}(f)(T^m_{\bar{\alpha}_1}(g)(\bar{x}_1),\bar{x}_2)=\\
&=& \Big(T^m_{\bar{\alpha}_1}(g)\circ T^n_{(r(g(\bar{\alpha}_1)),\bar{\alpha}_2)}(f) \Big) (\bar{x}).
\end{eqnarray*}

\end{proof}

Next, we will prove that if a compatible function $f$ on $\ab{A}$ with $f(I^{ar (f)}) \subseteq I$ can be generated from $f_1, \ldots, f_n \in \Comp (\ab{A})$, then
the restriction $f|_I$ of $f$ to $I$ can be generated by all translations $T^n_{\bar \alpha} (f_i)$ of the generators. All of these translations are
operations on $I$.
\begin{lem}\label{B}
Let $\A$ be a finite expanded group, and let $f_1,\ldots, f_m \in \Comp (\ab{A})$.
 Let $I$ be an ideal of $\A$, let $k:=|A/I|$, and let $(0=s_1,\dots,s_k)$ be a transversal modulo $I$.
Let $S := \{s_1, \ldots, s_k\}$, and let
\begin{equation*}
  B := \big\{ T^{ar(f_i)}_{(\alpha_1,\dots,\alpha_{ar(f_i)})}(f_i) \mid
i\in \{1,\dots,m\},
      (\alpha_1,\dots,\alpha_{ar(f_i)})\in \{s_1,\dots,s_k\}^{ar(f_i)} \big\} .
\end{equation*}
Then for every $f \in \Clo^{A}(f_1,\dots,f_m)$ we have:
\begin{equation} \label{eq:beta}
  \text{ for all } \vb{\beta} \in S^{ar (f)} \text{ : } T^{ar(f)}_{(\beta_1,\ldots, \beta_{ar (f)})} (f) \in \Clo^{A}( B ).
\end{equation}
\end{lem}
\begin{proof}
We will show that the set of functions $f \in \Comp (\ab{A})$ satisfying~\eqref{eq:beta}
is a subalgebra of $(\Comp (\ab{A}), \id_A, \zeta, \tau, \Delta, \nabla, \circ)$
that contains $\{f_1,\ldots, f_m\}$.
First, we show that all $f_i$ satisfy~\eqref{eq:beta}.
  Let  $i\in \{1,\dots,m\}$. Then $T^{ar(f_i)}_{(\beta_1,\ldots, \beta_{ar (f_i)})}(f_i)\in B \subseteq \Clo^{A} (B)$.
For showing that the functions that satisfy~\eqref{eq:beta} form a subalgebra,
we first show that the identity map $\id_A$ satisfies \eqref{eq:beta}.
 Let $\beta_1 \in S$. Then $T^{(1)}_{(\beta_1)} (\id_A) (x_1) = x_1 + \beta_1 - r (\beta_1)$. Since
 $\beta_1 \in S$, we have $r (\beta_1) = \beta_1$, and therefore
 $T^{(1)}_{(\beta_1)} (\id_A)  = \id_A$. Thus
 $T^{(1)}_{(\beta_1)} (\id_A) \in \Clo^A (B)$.
Now suppose that $g,h\in \Comp (\ab{A})$ both satisfy~\eqref{eq:beta}.
It follows from Lemma~\ref{12345} that for all $f\in \{\zeta g,\tau g,\Delta g,g\circ h,\nabla g\}$,
and for all $\vb{\beta} \in S^{ar(f)}$, we have
$T^{ar(f)}_{\vb{\beta}}(f)\in \Clo^{A} ( B )$.
 Hence the set of functions that satisfy~\eqref{eq:beta} is a subalgebra of
$(\Comp (\ab{A}), \id_A, \zeta, \tau, \Delta, \nabla, \circ)$ that contains
 $\{ f_1,\dots,f_m \}$. Hence every $f \in \Clo^{A}(f_1,\dots,f_m)$ satisfies~\eqref{eq:beta}.
\end{proof}

We will now investigate the connection between congruence preserving functions of $\ab{A}$ and $\AI$.
One construction that we need is the following:
\begin{de} \label{de:prosirenje}
   Let $\ab{A}$ be an expanded group, let $I$ be an ideal of $\ab{A}$, let $n \in \N$, let $f : I^n \to I$, and let $c\in I$.
   Then we  define $f^{c}:A^{n}\rightarrow I$ by
 \begin{equation}\label{prosirenje.def}
f^c(\bar{x}) =
\begin{cases}
f(\bar{x}), & \text{if } \bar{x} \in {I}^n,\\
c, & \text{else}.
\end{cases}
\end{equation}
\end{de}

For an expanded group $\ab{A}$, the algebra $\AI = (I, \{ f|_I \setsuchthat f \in \Comp (\ab{A}), \, f(I^{ar(c)}) \subseteq I \})$ also
has group operations among its fundamental operations; these are obtained
by restricting the group operations of $\ab{A}$ to $I$. We have the following easy relations between the congruences of $\AI$ and $\ab{A}$.
\begin{lem} \label{lem:comps}
   Let $\ab{A}$ be an expanded group, let $I$ be an ideal of $\ab{A}$, and let $(0=s_1,\ldots, s_k)$ be a transversal
   of $\ab{A}$ modulo $I$.
   \begin{enumerate}
      \item \label{it:1} The set of congruences of $\AI$ is given by
   \begin{equation} \label{eq:conai}
  \Con (\AI) = \{ \{ (x,y) \in I^2 \setsuchthat x - y \in J \} \setsuchthat J \text{ is an ideal of } \ab{A} \text{ with } J \subseteq I \}.
   \end{equation}
      \item \label{it:2} For each $n$-ary  congruence preserving function $f \in \Comp (\ab{A})$ and for all $\alpha_1,\ldots, \alpha_n \in \{s_1,\ldots, s_k\}$,
            the function $T^{n}_{(\alpha_1,\ldots,\alpha_n)} (f)|_{I^n}$ is a congruence preserving function of $\AI$.
      \item \label{it:3} If the ideal $I$ cuts the lattice of ideals of $\ab{A}$, $c \in I$, and $g$ is an $n$-ary  congruence preserving function of $\AI$, then the function
            $g^c$
is a congruence preserving function of $\A$.
   \end{enumerate}
\end{lem}
\emph{Proof:}
   \eqref{it:1} Since all fundamental operations of $\AI$ are restrictions of congruence preserving operations
     of $\A$, each element of the right hand side of \eqref{eq:conai} is a congruence of $\AI$.
    Now let $\gamma$  be a congruence relation of $\AI$. We will first show that $0/\gamma$ is an ideal of $\ab{A}$.
    To this end, let $n \in \N$, let $c$ be an $n$-ary fundamental operation of $\ab{A}$, let $a_1,\ldots, a_n \in A$, and let
    $i_1,\ldots, i_n \in 0/\gamma$.
    Then the function $c' (x_1, \ldots, x_n) := c (a_1 + x_1, \ldots, a_n + x_n) - c(a_1, \ldots, a_n)$ is a congruence
    preserving function of $\ab{A}$ with $c(I^n) \subseteq I$, and therefore its restriction to $I$ is a fundamental operation of
    $\AI$. Thus $c'(0,\ldots, 0) \equiv_\gamma c'(i_1,\ldots, i_n)$, which implies that
    $c(a_1 + i_1, \ldots, a_n + i_n) - c (a_1, \ldots, a_n) \in 0/\gamma$. Setting first $c (x,y) := x + y$, and then
    $c(x) =  - x$, we obtain
     that $0 / \gamma$ is a subgroup of $(A, +)$, setting $c (x) := x$, we obtain that $0/ \gamma$ is normal,
    and letting $c$ be an arbitrary fundamental operation of $\A$, we obtain the ideal property.
    Hence $0 / \gamma$ is an ideal of $\ab{A}$. Using that $\AI$ has group operations among its fundamental operations,
     it is now easy to see that
    $\gamma = \{ (x,y) \in I^2 \setsuchthat x- y \in 0/ \gamma\}$, and hence $\gamma$ appears in the right hand side of \eqref{eq:conai}.
  Item \eqref{it:2}~now follows from the description of congruences in~\eqref{it:1}.
  For item~\eqref{it:3}, we let $J$ be an ideal of $\A$ and $\vb{x}, \vb{y} \in A^n$ with
  $\vb{x} \equiv_J \vb{y}$. Since $I$ cuts the ideal lattice,
   $J \le I$ or $J \ge I$. In the first case, the definition of $f$ and the fact
   that $g$ is congruence preserving yield $f(\vb{x}) - f(\vb{y}) \in J$. If $J \ge I$,
   then $f(\vb{x}) - f(\vb{y}) \in J$ holds because the range of $f$ is contained in $I$. \qed

\begin{lem}\label{Arestrik}
Let $\A$ be a finite expanded group, and let $I\in \mathrm{Id}(\A)$ be an ideal that cuts the ideal lattice
of $\A$.
If $\mathrm{Comp}(\A)$ is finitely generated, then $\mathrm{Comp}(\AI)$ is also finitely generated.
\end{lem}

\begin{proof}
Assume that $\mathrm{Comp}(\mathbf{A}) = \Clo^{A} (f_1,\dots , f_m )$ and that $(0=s_1,s_2,\dots ,s_k)$, where $k=|A/I|,$ is a transversal modulo $I$.
Let
\begin{equation*}
   \begin{array}{rcl}
        B & := &   \big\{ T^{ar(f_i)}_{(\alpha_1,\dots,\alpha_{ar(f_i)})}(f_i) \mid
i\in \{1,\dots,m\},
            (\alpha_1,\dots,\alpha_{ar(f_i)})\in \{s_1,\dots,s_k\}^{ar(f_i)} \big\}, \\
       C & :=   & \{g|_{I}: g\in B\}.
   \end{array}
\end{equation*}
We will prove that $\Clo^{I} ( C )=\mathrm{Comp}(\AI)$.

For $\subseteq$, it is sufficient to prove
$C\subseteq \mathrm{Comp}(\AI)$.
Let $g := T^{n}_{(\alpha_1,\dots,\alpha_{n})}(f_i)|_{I}\in C$, where $n=ar(f_i)$. Then from Lemma~\ref{lem:comps}, we
see that $g \in \Comp (\AI)$.

Next,  we show $\mathrm{Comp}(\AI)\subseteq\Clo^{I} ( C )$. Let $f\in \mathrm{Comp}(\AI)$ and $n=ar(f)$.
Let $f^0 := f \cup 0_{(A^{n} \setminus I^{n})}$; hence $f^0$ interpolates $f$ on $I^n$ and
is $0$ everywhere else. By Lemma~\ref{lem:comps}, $f^0 \in \Comp (\ab{A})$.
Thus by the assumptions, $f^0\in \mathrm{Comp}(\A)=\Clo^{A} (f_1,\dots , f_m)$, and hence from Lemma \ref{B}, we have
$
T^n_{(0,\dots,0)}(f^0)\in \Clo^{A}( B )$.
Since $f^0({I}^n)\subseteq I$, we have
$f^0  = T^n_{(0,\dots,0)}(f^0)$, and therefore $f^0 \in \Clo^{A} (B)$.
Let \[ D := \{ f \in \Comp (\ab{A}) \setsuchthat f (I^{ar(f)}) \subseteq I \}.
\]
 Then $D$ is a clone.
 Now $\{f^0\} \cup B \subseteq D$. Let $P_I$ be the set of all finitary operations on $I$.
The restriction mapping $\varphi : D \to P_I$, $f \mapsto f|_I$ is a homomorphism from
the function algebra $(D, \id_A, \zeta, \tau, \Delta, \nabla, \circ)$ to
$(P_I, \id_I, \zeta, \tau, \Delta, \nabla, \circ)$. Hence from $f^0 \in \Clo^{A} ( B )$, we obtain
$\varphi (f^0) \in \Clo^{I} (\varphi (B) )$, and thus by Lemma~\ref{lem:clohom}, $f \in \Clo^{I} ( C )$.
\end{proof}

\section{Lifting generators from homomorphic images}
Our task in this section is to show that if $\Comp (\ab{A}/I)$ and $\Comp (\AI)$ are both finitely generated,
then so is $\Comp (\ab{A})$. This amounts to producing generators for $\Comp (\ab{A})$ from the generators
of the compatible functions of $\ab{A}/I$ and $\AI$.
To this end, we define certain  modifications of the projection operations.
Let $\A$ be an expanded group whose ideal lattice has a cutting element $I$.
For every $m,n\in \mathbb{N}$, $m\leq n$, we define a function $\0nmI:A^n\rightarrow I$ by
\begin{equation}\label{o(j-na)}
\0nmI(x_1,\dots,x_n):=
\begin{cases}
x_m, & \text{if } (x_1,\dots,x_n) \in {I}^n,\\
0, & \text{else}.
\end{cases}
\end{equation}
By Proposition~\ref{pro:de} (where $\delta = \eps$ is set to be the congruence induced by $I$),
all $\0nmI$ are congruence preserving. We will now see that one modified binary projection
generates these modified projections of all arities.

\begin{lem}\label{o}
Let $\A$ be an expanded group whose ideal lattice has a cutting element $I$. Then for every $m,n\in \mathbb{N}$,
$m\leq n$, we have that $\0nmI\in \Clo^{A} ( \021I )$.
\end{lem}

\begin{proof}
We prove the statement by induction on the number $n$. For $n=1$ we have $\011I=\Delta \021I\in \Clo^{A} (\021I )$.
For the induction step, we let $n \ge 2$ and
assume that $\0{n-1}mI\in \Clo^{A} (\021I)$ for every $m\in\{1,\dots,n-1\}$.
We will verify that
$\0nmI\in \Clo^{A} (\021I)$ for every $m\in\{1,\dots,n\}$. To this end, we first
observe that $\022I = \zeta \021I$ and prove that
$$\0niI=\0{n-1}iI\circ \021I, \quad \mbox{for }i\in\{1,\dots, n-1\},$$
$$\mbox{and }\quad \0nnI=\zeta(\0{n-1}{n-1}I\circ \zeta \022I),$$
or equivalently,
$$\0niI(x_1,\dots,x_n)=\021I(\0{n-1}iI(x_1,\dots, x_{n-1}),x_n), \quad \mbox{for }i\in\{1,\dots, n-1\},$$
$$\mbox{and }\quad \0nnI(x_1,\dots,x_n)=\022I(x_1,\0{n-1}{n-1}I(x_2,\dots,x_n)),$$
for all $x_1,\dots ,x_n\in A.$
Let $i\in\{1,\dots, n-1\}$.
\begin{itemize}
\item If $(x_1,\dots,x_n) \in {I}^n$, then we have
$$\021I(\0{n-1}iI(x_1,\dots, x_{n-1}),x_n)=\021I(x_i,x_n)=x_i=\0niI(x_1,\dots,x_n).$$
\item Otherwise, there exists a $j\in \{1,\dots,n\}$ such that $x_j\not\in I.$
 \begin{itemize}
 \item Let $j\in\{1,\dots, n-1\}$. Note that $\021I(0,x_n)=0$ for every $x_n\in A$.
 Then we obtain $$\021I(\0{n-1}iI(x_1,\dots, x_{n-1}),x_n)=\021I(0,x_n)=0=\0niI(x_1,\dots,x_n).$$
 \item For $j=n$, we have $\021I(\0{n-1}iI(x_1,\dots, x_{n-1}),x_n)=0=\0niI(x_1,\dots,x_n)$.
 \end{itemize}
\end{itemize}

Now assume that $i=n$.
\begin{itemize}
\item The case when $(x_1,\dots,x_n) \in {I}^n$ is same as above.
\item Let $x_j\not\in I$ for a $j\in \{1,\dots,n\}$.
 \begin{itemize}
 \item Let $j\in\{2,\dots, n\}$.
 Note that $\022I(x_1, 0)=0$ for every $x_1\in A$.
 Then we obtain $$\022I(x_1,\0{n-1}{n-1}I(x_2,\dots,x_n))=\022I(x_1, 0)=0=\0nnI(x_1,\dots,x_n).$$
 \item For $j=1$, we have $\022I(x_1,\0{n-1}{n-1}I(x_2,\dots,x_n))=0=\0nnI(x_1,\dots,x_n)$.
 \end{itemize}
\end{itemize}
\end{proof}
In the following lemmas we use the functions of the form $f^{f(\bar{0})}$ that were defined in~Definition~\ref{de:prosirenje}.
 Hence $f^{f (\bar{0})}|_{I^{ar(f)}} = f$ and $f^{f(\bar{0})} (\vb{x}) = f (\vb{0})$ for
all $\vb{x} \in A^{ar(f)} \setminus I^{ar(f)}$.
\begin{lem}\label{nula.u.nulu}
Let $\A$ be an expanded group whose ideal lattice has a cutting element $I$,
and let $g_1, \ldots, g_s \in \Comp (\AI)$.
Then for all $f \in \Clo^{I} (g_1,\dots ,g_s)$
we have  $f^{f(\bar{0})}\in \Clo^{A} (\{ {g_1}^{g_1(\bar{0})},\dots,{g_s}^{g_s(\bar{0})} \} \cup \{\021I\} )$.
\end{lem}
\begin{proof}
Let
\[
    B := \big\{ f \in \bigcup_{n \in \N} I^{I^n} \setsuchthat f^{f(\vb{0})} \in \Clo^{A} (\{ {g_1}^{g_1(\bar{0})},\dots,{g_s}^{g_s(\bar{0})} \} \cup \{\021I\} ) \big\}.
\]
We will first show that $B$ is a subalgebra of
$\algop{\Comp (\AI)}{\id_I, \zeta, \tau, \Delta, \nabla, \circ}$.
We start by verifying that $(\id_I)^0$ lies in
 $\Clo^{A} (\{ {g_1}^{g_1(\bar{0})},\dots,{g_s}^{g_s(\bar{0})} \} \cup \{\021I\} )$.
 Since $(\id_I)^0 = \011I$, this is a consequence of Lemma~\ref{o}. Therefore, $\id_I \in B$.
  For proving that $B$ is closed under all operations, we
  let $g, h \in B$, and we show that all $f \in \{ \zeta h, \tau h, \Delta h, \nabla h, g \circ h \}$
  satisfy $f \in B$.
  To this end, we will prove
  \begin{enumerate}
     \item \label{it:gl1}
               ${(\zeta h)}^{(\zeta h)(\bar{0})}= \zeta (h^{h(\bar{0})})$,
     \item \label{it:gl2}
               ${(\tau h)}^{(\tau h)(\bar{0})}= \tau (h^{h(\bar{0})})$,
     \item \label{it:gl3}
               ${(\Delta h)}^{(\Delta h)(\bar{0})}= \Delta (h^{h(\bar{0})})$,
    \item \label{it:gl4}
               ${(\nabla h)}^{(\nabla h)(\bar{0})}= \nabla (h^{h(\bar{0})})$,
     \item \label{it:gl5} $(g\circ h)^{g\circ h (\bar{0})}= (g^{g(\bar{0})}\circ h^{h(\bar{0})})
                       (\0{m+n-1}1I,\dots,\0{m+n-1}{m+n-1}I)$.
 \end{enumerate}
Let $x_1,\dots,x_{m+n-1}\in A$.
For item~\eqref{it:gl1}, we compute
\begin{eqnarray*}
& & {(\zeta h)}^{(\zeta h)(\bar{0})}(x_1,x_2,\dots,x_n)=\\
&=& \begin{cases}
 {(\zeta h)}(x_1,x_2,\dots,x_n), & \text{if }  (x_1,x_2,\dots,x_n)\in {I}^n,\\
 (\zeta h)(0,\dots,0), & \text{else}
 \end{cases}\\
&=& \begin{cases}
 h(x_2,\dots,x_n,x_1), & \text{if }  (x_1,x_2,\dots,x_n)\in {I}^n,\\
 h(0,\dots,0), & \text{else}
 \end{cases}\\
&=& h^{h(\bar{0})}(x_2,\dots,x_n,x_1)\\
&=& \zeta(h^{h(\bar{0})}) \, (x_1,x_2,\dots,x_n).
\end{eqnarray*}
Analogously we obtain items \eqref{it:gl2}, \eqref{it:gl3} and
\eqref{it:gl4} substituting $\zeta$ by $\tau$, $\Delta$ and $\nabla$
(and adjusting the variables).
 For item \eqref{it:gl5}
 we let $\bar{x}=(x_1,x_2,\dots,x_{m+n-1}),$
$\bar{x}_1=(x_1,x_2,\dots,x_m)$ and $\bar{x}_2=(x_{m+1}, x_{m+2},
\dots,x_{m+n-1}).$ Then
\begin{eqnarray*}
& & (g^{g(\bar{0})}\circ h^{h(\bar{0})})(\0{m+n-1}1I,\dots,\0{m+n-1}{m+n-1}I)(\bar{x})\\
&=& (g^{g(\bar{0})}\circ h^{h(\bar{0})})(\0{m+n-1}1I(\bar{x}),\dots,\0{m+n-1}{m+n-1}I(\bar{x}))\\
&=& \begin{cases}
 (g^{g(\bar{0})}\circ h^{h(\bar{0})})(\bar{x}), & \text{if }  \bar{x}\in {I}^{m+n-1},\\
 (g^{g(\bar{0})}\circ h^{h(\bar{0})})(\bar{0}), & \text{else}
 \end{cases}\\
&=& \begin{cases}
 h^{h(\bar{0})}(g^{g(\bar{0})}(\bar{x}_1),\bar{x}_2), & \text{if }  \bar{x}\in {I}^{m+n-1},\\
 h^{h(\bar{0})}(g^{g(\bar{0})}(\bar{0}),\bar{0}), & \text{else}
 \end{cases}\\
&=& \begin{cases}
 h^{h(\bar{0})}(g(\bar{x}_1),\bar{x}_2), & \text{if }  \bar{x}\in {I}^{m+n-1},\\
 h^{h(\bar{0})}(g(\bar{0}),\bar{0}), & \text{else}
 \end{cases}\\
&=& \begin{cases}
 h(g(\bar{x}_1),\bar{x}_2), & \text{if }  \bar{x}\in {I}^{m+n-1},\\
 h(g(\bar{0}),\bar{0}), & \text{else}
 \end{cases}\\
&=& \begin{cases}
 (g\circ h)(\bar{x}), & \text{if }  \bar{x}\in {I}^{m+n-1},\\
 (g\circ h)(\bar{0}), & \text{else}
 \end{cases}\\
&=& (g\circ h)^{(g\circ h) (\bar{0})}(\bar{x}).
\end{eqnarray*}

 By assumption, $g^{g(\vb{0})}$ and $h^{h(\vb{0})}$ are both elements of $\Clo^{A} (\{ {g_1}^{g_1(\bar{0})},\dots,{g_s}^{g_s(\bar{0})} \} \cup \{\021I\} )$.
  Now the five equations given above and Lemma~\ref{o} imply
  $f^{f(\vb{0})} \in \Clo^{A} (\{ {g_1}^{g_1(\bar{0})},\dots,{g_s}^{g_s(\bar{0})} \} \cup \{\021I\} )$.

  This completes the proof that $B$ is a subalgebra. It is easily seen that it contains
  $g_1, \ldots, g_s$. Thus $\Clo^{I} (g_1, \ldots, g_s) \subseteq B$, which implies the result.
\end{proof}
By
$C(I)$, we will denote the set of all constant functions on $I$, and $C(I)^0 := \{c^0\mid c\in C(I)\}.$

\begin{pr}\label{prosirenja}
Let $\A$ be an expanded group whose ideal lattice has a cutting element $I$.
If $\mathrm{Comp}(\AI)=\Clo^{I} ( g_1,\dots ,g_s)$ with $s\in\mathbb{N}$ and if $f\in \mathrm{Comp}(\AI)$,
then $f^0\in \Clo^{A} ( \{ {g_1}^{g_1(\bar{0})},\dots,{g_s}^{g_s(\bar{0})} \} \cup \{\021I,+\} \cup C(I)^0 )$.
\end{pr}

\begin{proof}
Let $f_T:I^{ar(f)}\rightarrow I$ be defined by $f_T(\bar{x}):= f(\bar{x})-f(\bar{0})$.
Note that $f_T(\bar{0})=0$ and $f_T\in \mathrm{Comp}(\AI)=\Clo^{I} (g_1,\dots ,g_s)$.
By
Lemma \ref{nula.u.nulu},
 we know that
$$f_T^{0}\in \Clo^{A} (\{ {g_1}^{g_1(\bar{0})},\dots,{g_s}^{g_s(\bar{0})} \} \cup \{\021I\}).$$

From $f(\bar{x})=f_T(\bar{x})+f(\bar{0})$ follows that ${f}^0(\bar{x})={f_T}^0(\bar{x})+(f(\bar{0}))^0$.
Therefore we obtain $f^0\in \Clo^{A} ( \{ {g_1}^{g_1(\bar{0})},\dots,{g_s}^{g_s(\bar{0})} \} \cup \{\021I,+\} \cup C(I)^0).$
\end{proof}

\begin{thm} \label{thm:cut}
    Let $\ab{A}$ be a finite expanded group, and let $I$ be a cutting element
    of the ideal lattice of $\ab{A}$. Then
    the clone $\Comp (\ab{A})$ is finitely generated if and only if
    both clones $\Comp (\ab{A} / I)$ and
    $\Comp (\AI)$ are finitely generated.
\end{thm}
\emph{Proof:}
    The ``only if''-part has been proved in Lemmas~\ref{lem:factor}~and~\ref{Arestrik}.
    For the ``if''-part, we assume that $\Comp (\AI)$ and $\Comp (\ab{A} / I)$ are
    finitely generated, and we will produce a finite set of generators
    of $\Comp (\ab{A})$.
Let
$G_1$ be a  finite set of generators of the clone
$\mathrm{Comp}(\mathbf{A}/I)$,  let $(0=s_1,s_2,\dots ,s_k)$ be a transversal modulo $I$,
where $k=|A/I|$,
and let  $G_2:=\{\tilde{g}:g\in G_1\}$, where $\tilde{g}$ is defined as in
Proposition~\ref{pro:cut1}. By this proposition, $G_2 \subseteq \mathrm{Comp}(\mathbf{A})$.

Now we will show that for every $f\in \mathrm{Comp}(\mathbf{A})$ there is a $g\in \Clo^{A} (G_2)$ such that
\begin{equation} \label{eq:fg}
  \text{for all } \bar{x}=(x_1,\dots ,x_{ar(f)})\in A^{ar(f)} \,: \, f(\bar{x})\equiv_{I} g(\bar{x}).
 \end{equation}
 We know that $\Clo^{A/I} (G_1) = \Comp (\A /I)$.
Let $f$ be an arbitrary compatible function on $\mathbf{A}$. Then,
we have that $f^{\mathbf{A}/I}\in
\mathrm{Comp}(\mathbf{A}/I)=\Clo^{A/I} ( G_1 )$. Now let $\varphi :
\Comp (\ab{A}) \to
  \Comp (\A / I)$, $h \mapsto h^{\ab{A}/I}$. Then $G_1 = \varphi (G_2)$.
   Hence Lemma~\ref{lem:clohom} yields
   $\varphi (\Clo^A (G_2)) = \Comp (\A / I)$.
   Since $\varphi (f) \in \Comp (\A / I)$, we have
   $\varphi (f)  \in \varphi (\Clo^A (G_2))$. Therefore, there is
   $g \in \Clo^A (G_2)$ such that $g^{\ab{A}/I} = f^{\ab{A}/I}$, which completes the construction of
   a function $g$ satisfying~\eqref{eq:fg}.

Since $h :=f-g$ is a function with $h(A^{ar(f)}) \subseteq I$ and $f=h+g$, our aim is now to
find a finite set of generators of all compatible functions with range contained
in $I$. Let $n \in \N$.
For every $\bar{\alpha}=(\alpha_1,\dots,\alpha_n)\in \{s_1,\dots,s_k\}^n$ we introduce a
function $h_{\bar{\alpha}}':A^n\rightarrow I$ by
\begin{equation*}
h_{\bar{\alpha}}'(\bar{x}) =
\begin{cases}
h  (\bar{x}+\bar{\alpha}), & \text{if } \bar{x} \in {I}^n,\\
0, & \text{else}.
\end{cases}
\end{equation*}
Then for all $\bar{x}\in A^n$ we obtain
\[
   h(\bar{x})=\sum_{\bar{\alpha}\in \{s_1,\dots,s_k\}^n} h_{\bar{\alpha}}'(\bar{x}-\bar{\alpha}).
\]
 Since $h$ is a congruence preserving function with range contained in $I$,
 Proposition~\ref{pro:de} yields that for every
  $\bar{\alpha}=(\alpha_1,\dots,\alpha_n)\in \{s_1,\dots,s_k\}^n$, the function $h_{\bar{\alpha}}'$
  is compatible on $\mathbf{A}$.
  Now for every $\bar{\alpha}=(\alpha_1,\dots,\alpha_n)\in \{s_1,\dots,s_k\}^n$, we have that $h_{\bar{\alpha}}'({I}^n)\subseteq I$, $h_{\bar{\alpha}}'(A^n\setminus {I}^n)=\{0\}$ and $h_{\bar{\alpha}}'|_{I}\in \mathrm{Comp}(\ABar|_{I})$. Let $G_3$ be a finite set of functions such that $\Clo^{I} (G_3)=\mathrm{Comp}(\ABar|_{I})$ and let $G_3^0=\{g^{g(\bar{0})} \mid g\in G_3\}$, where $g^{g(\bar{0})}$ is defined as in (\ref{prosirenje.def}).
Since $h_{\bar{\alpha}}'=(h_{\bar{\alpha}}'|_{I})^0$ and  $h_{\bar{\alpha}}'|_{I}\in \mathrm{Comp}(\ABar|_{I})=\Clo^{I} ( G_3 )$, Proposition \ref{prosirenja} implies that
$$h_{\bar{\alpha}}' \in \Clo^{A} (G^0_3 \cup \{\021{I},+\} \cup C(I)^0).$$
Therefore, $h \in \Clo^{A} (G_3^0 \cup \{\021{I},+\}\cup C(I)^0 \cup \{x-s_i\mid i\in \{1,\dots ,k\}\} )$.
 Thus
 \[ \mathrm{Comp}(\mathbf{A})=\Clo^{A} (G_2\cup G_3^0 \cup \{\021{I},+\}\cup C(I)^0 \cup \{x-s_i\mid i\in \{1,\dots ,k\}\} ),
 \]
 and so we have found a finite set of generators of $\Comp (\ab{A})$. \qed

\section{The proofs of the theorems}\label{sec:mainproof}
\emph{Proof of Theorem \ref{thm:oppositewithoutM2}:} We proceed by
induction on $n$. For $n=1$, we assume that $\IntI [{\mathbf
0},S_1]=\mathrm{Id}(\mathbf{A})$ is a two-element lattice or that
$\mathrm{Id}(\mathbf{A})$ does not split. If
$|\mathrm{Id}(\mathbf{A})|=2$, then $\ab{A}$ is simple and hence
every finitary operation on $\mathbf{A}$ is compatible. Since every
operation is a composition of binary operations (cf. \cite[Theorem
3.1.6]{KP:PCIA}), the binary compatible functions then generate all
(compatible) functions.
 If
$|\mathrm{Id}(\mathbf{A})|>2$ and $\mathrm{Id}(\mathbf{A})$ does not
split, then we can apply
 Lemma~\ref{Malcev21} because each expanded group has a Mal'cev term and obtain that the clone $\mathrm{Comp}(\mathbf{A})$ is
finitely generated.

 For the induction step, let $n \ge 2$.
 Since $\IntI [{\mathbf 0},S_1]\cong \mathrm{Id}(\ABar|_{S_1})$ and $\IntI [S_1,{\mathbf 1}]\cong \mathrm{Id}(\A/S_1)$,
the induction hypothesis yields that $\mathrm{Comp}(\ABar|_{S_1})$
and $\mathrm{Comp}(\mathbf{A}/S_1)$ are finitely generated.
Now Theorem~\ref{thm:cut} implies that $\Comp (\ab{A})$ is finitely generated. \qed

\emph{Proof of Theorem \ref{teorema}:}
$(1)\Rightarrow (2)$: Assume that $\mathrm{Comp}(\A)$ is finitely generated.
We prove the statement by induction on the number $n$ of nonzero elements that cut the lattice.

In the case $n=1$, the interval $\IntI [{\mathbf 0},S_1]$ is the
whole lattice $\mathrm{Id}(\mathbf{A})$, and thus by Lemma
\ref{lema1}, $\mathrm{Id} (\A)$ is simple. Now
Theorem~\ref{thm:malcev} yields that $\mathrm{Id}(\mathbf{A}) =
\IntI [{\mathbf 0}, S_1]$ is of the required form.

 For the induction step, let $n \ge 2$.
 By Lemma \ref{Arestrik}, $\mathrm{Comp}(\ABar|_{S_1})$ is finitely generated,
and we can apply the induction hypothesis on the expanded group $\ABar|_{S_1}$.
 We notice  $\mathrm{Id}(\ABar|_{S_1})\cong \IntI [{\mathbf 0},S_1]$.
Thus, in the lattice $\mathrm{Id}(\ABar|_{S_1})\cong \IntI [{\mathbf
0},S_1]$ we have either ${\mathbf 0} \prec S_1$ or there is no
splitting pair.
 Now we consider $\mathbf{A}/S_1$. By Lemma~\ref{lem:factor},
 $\Comp (\A / S_1)$ is finitely generated.
Since the ideal lattice of $\mathbf{A}/S_1$ is isomorphic to $\IntI
[S_1,{\mathbf 1}]$,
 the induction hypothesis yields that the ideal lattice of $\ab{A}$ also has the required
 properties above $S_1$.
 (2)$\Rightarrow$(1): Theorem~\ref{thm:oppositewithoutM2}. \qed

We notice that in the implication ~\eqref{it:t1}$\Rightarrow$~\eqref{it:t2} of
Theorem~\ref{teorema},
we cannot omit the assumption that the ideal lattice of $\ab{A}$ is
$\ob{M}_2$-free. To this end, let $\ab{F}$ be a finite field.
The clone of congruence preserving operations of $\ab{A} := \ab{F} \times \ab{F}$
is finitely generated by Proposition~\ref{pro:dp2}. The ideal lattice
of $\ab{A}$ is isomorphic to $\ob{M}_2$ and hence~\eqref{it:t2} does not hold.

{\noindent {\it Proof of Theorem \ref{thm:pgroups}:}}
Since the lattice of normal subgroups of a $p$-group is $\ob{M}_2$-free, the result
is a consequence of Theorem \ref{teorema}. \qed

\section{Acknowledgement}

We thank J.\ Farley, K. Kaarli, and C.\ Pech for fruitful discussions on parts of this paper, and the referee for numerous 
valuable suggestions.

\def\cprime{$'$}
\providecommand{\bysame}{\leavevmode\hbox to3em{\hrulefill}\thinspace}
\providecommand{\MR}{\relax\ifhmode\unskip\space\fi MR }
\providecommand{\MRhref}[2]{%
  \href{http://www.ams.org/mathscinet-getitem?mr=#1}{#2}
}
\providecommand{\href}[2]{#2}

\end{document}